\documentclass[dvipsaoas,preprint]{imsart}

\RequirePackage[T1]{fontenc}
\RequirePackage{amsthm,amsmath,amsfonts}
\RequirePackage[numbers]{natbib}
\RequirePackage[colorlinks,citecolor=blue,urlcolor=blue]{hyperref}

\startlocaldefs
\numberwithin{equation}{section}
\theoremstyle{plain}
\newtheorem{thm}{Theorem}[section]
\newtheorem{lem}{Lemma}[section]

\theoremstyle{definition}
\newtheorem{defn}{Definition}[section]
\newtheorem{exmp}{Example}[section]

\def\lw{{\ell_1}}
\def\lt{{\ell_2}}
\def\linf{{\ell_\infty}}

\endlocaldefs

\begin{document}

\begin{frontmatter}
\title{Finite sample posterior concentration in high-dimensional regression}

\runtitle{Concentration in high-dimensional regression}


\begin{aug}
\author{\fnms{Nate} \snm{Strawn}\ead[label=e1]{nstrawn@math.duke.edu}},
\author{\fnms{Artin} \snm{Armagan}\ead[label=e2]{artin.armagan@sas.com}},
\author{\fnms{Rayan} \snm{Saab}\thanksref{t3}\ead[label=e3]{rayans@math.duke.edu}},
\author{\fnms{Lawrence} \snm{Carin}\ead[label=e4]{lcarin@ee.duke.edu}},
\and
\author{\fnms{David} \snm{Dunson}\ead[label=e5]{dunson@stat.duke.edu}}

\thankstext{t3}{R. Saab was supported in part by a Banting Postdoctoral Fellowship, administered by the Natural Science and Engineering Research Council of Canada.}
\runauthor{N. Strawn, A. Armagan, R. Saab et al.}

\affiliation{Duke University}

\address{Department of Mathematics \\
Duke University\\
\printead{e1,e3}}

\address{SAS Institute, Inc. \\
\printead{e2}}

\address{Department of Electrical and Computer Engineering\\
Duke University\\
\printead{e4}}

\address{Department of Statistical Sciences\\
Duke University\\
\printead{e5}}

\end{aug}

\begin{abstract}
We study the behavior of the posterior distribution in high-dimensional Bayesian Gaussian linear regression models having $p\gg n$, with $p$ the number of predictors and $n$ the sample size.  Our focus is on obtaining quantitative finite sample bounds ensuring sufficient posterior probability assigned in neighborhoods of the true regression coefficient vector ($\beta^0$) with high probability. We assume that $\beta^0$ is approximately $S$-sparse and obtain universal bounds, which provide insight into the role of the prior in controlling  concentration of the posterior. Based on these finite sample bounds, we examine the implied asymptotic contraction rates for several examples showing that sparsely-structured and heavy-tail shrinkage priors exhibit rapid contraction rates. We also demonstrate that a stronger result holds for the sparsity($S$)-Gaussian\footnote[2]{A binary vector of indicators ($\gamma$) is drawn uniformly from the set of binary sequences with exactly $S$ ones, and then each $\beta_i\sim\mathcal{N}(0,V^2)$ if $\gamma_i=1$ and $\beta_i=0$ if $\gamma_i=0$.} prior. These types of finite sample bounds provide guidelines for designing and evaluating priors for high-dimensional problems.
\end{abstract}



\end{frontmatter}

\section{Introduction}
Consider the Gaussian linear regression model
\begin{eqnarray}
y_i=x_i^T\beta^0+e_i,\quad e_i \sim \mathcal{N}(0,\sigma^2),\quad i=1,\ldots,n,  \label{eq:base}
\end{eqnarray}
where $x_i$ and $\beta^0$ are $p$-dimensional vectors, and the noise variance $\sigma^2$ is known.  In modern applications, it has become desirable to collect high-dimensional data in which the sample size $n$ is much smaller than the number of predictors (or variables) $p$. For example, there is great interest (see \cite{GBK08,LDP07}) in reducing MRI scan times (hence small $n$ is desired) while still producing high-resolution images of excellent quality ($\widehat{\beta}$ very close to $\beta^0$). In the MRI setting, $\beta^0$ represents the characteristics of the entity being sensed ($e.g.$, material properties of a human body), and $x_i$ represents the $i$th projection-class measurement, where often in MRI $x_i$ corresponds to a particular Fourier projection, and hence $y_i$ is a (noisy) Fourier component of $\beta^0$. The vector $\beta^0$ may be sparse or nearly sparse in its native (pixel) basis, or alternatively $\beta^0$ may represent the (near-sparse) wavelet coefficients of the entity being sensed.

 A recent surge of research in the estimation-theoretic properties for (\ref{eq:base}) has established powerful techniques and groundbreaking results. Much of the theory supporting these advancements is predicated on a \emph{sparsity} (or compressibility) assumption for the true $\beta^0$. Supposing that $\beta^0$ is approximately $S$-sparse (that is, $\beta^0$ can be well approximated by a vector with $S$ nonzero entries), the theory of compressed sensing tells us that (conditional on a bound for the noise term) we can compute a close estimate $\widehat{\beta}$ to $\beta^0$ using only
\[
n\geq C S\log(p/S)
\]
samples (see \cite{CP07, CRT05, CT05, DMM09, FoucartRauhut2013} among many others). Reasonable bounds for the constant $C$ can be found in the works \cite{BDDW08, FoucartRauhut2013, RV08}. This field is termed ``compressed sensing,'' because rather than directly measuring the $p$ values of $\beta^0$, we perform $n\ll p$ linear measurements on all the values (encoding), and then computationally attempt to reconstruct $\beta^0$ from these measurements (decoding). These bounds establish the accuracy of reconstructing $\beta^0$ based on the $n\ll p$ compressive linear measurements. This general class of problems is also of interest in linear regression models, apart from compressed sensing; there, $x_i$ may represent a covariate vector and $\beta^0$ is the (assumed sparse or near-sparse) regression vector that maps the covariates to a noisy data sample $y_i$.

The estimators used in compressed sensing are generally motivated by relaxation of the usual penalties (e.g. \cite{Aka74} and variants) used in model selection, and hence convex optimization \cite{CP07, CRT05, CT05, Tro06} or iterative methods \cite{BI09, DMM09, NT09, TG07} produce solutions efficiently. The precise error bounds can be made explicit using metric geometry arguments  (as in \cite{CP07, CRT05, CT05, SY10}), and an asymptotic theory (see \cite{BD10,DMM09, DT05, RFG12}) reveals the existence of explicit phase transitions for practical recovery of sparse $\beta^0$. Results of this nature provide us with a statistical power calculation (calculation of the error bounds that hold with high probability given a sample of size $n$) whenever we have a reasonable estimate for the true sparsity or compressibility. 

For the example of MRI sensing, we can use numerous examples (e.g. many different patients' imaged anatomy), and then estimate a bound for the compressibility of this collected data in a (for example) wavelet basis (for more on wavelets, see \cite{Mallat08}). With knowledge of the anticipated (near) sparsity level of $\beta^0$ ($i.e.$, bound on $S$), the aforementioned theoretical results provide guidelines for the number of MRI scans $\{y_i\}_{i=1,n}$ required to image a new patient.

From the viewpoint of model selection and regression, one may estimate the sparsity level using the method in \cite{Lopes12}, and the determination of the sparsity level $S$ represents an upper bound on the complexity of a satisfactory model. In this setting, compressed sensing tells us that only $C S\log(p/S)$ measurements are required to closely approximate an $S$-sparse model contingent upon the fact that there is an $S$-sparse model ($\beta^0$) which explains the data in a satisfactory manner. From yet another perspective, if allowed $n\ll p$ samples, we may compute the largest $S$ such that $n\geq CS\log(p/S)$, which then guarantees that we will approximately recover a model of complexity $\leq S$ if one exists. Efficient cross validation can also be performed using the method in \cite{Ward2009}.

The sparsity assumption effectively imposes an improper prior (a distribution which is not a probability distribution) on $\beta$. This improper prior is a singular distribution on the $\binom{p}{S}$ planes in $\mathbb{R}^p$ whose members always have at most $S$ nonzero entries (the $S$-sparse planes), and the measure of a Borel set $\mathcal{A}\subset\mathbb{R}^p$ under this distribution is exactly the sum of the $S$-dimensional areas of the regions obtained by intersecting $\mathcal{A}$ with each $S$-sparse plane. In this framework, we see that the approach of compressed sensing coincides with maximum \emph{a posteriori} (MAP) estimation arising from this prior and the model (\ref{eq:base}). Solving this optimization problem requires the examination of all $\binom{p}{S}$ potential supports (lists of the nonzero entries) for the true parameter $\beta^0$. To circumvent this computational difficulty, a relaxed problem involving minimizing $\|\beta\|_1$ subject to $X\beta = y$ (or variants of this problem) may be solved. Such an approach admits an interpretation as a MAP estimation under a Laplace prior:
\[
\Pi_{\text{Lap}}(\beta)=\left(\frac{\lambda}{2}\right)^p\prod_{i=1}^p\exp\{-\lambda\vert\beta_i\vert\}
\]
The resulting MAP estimator for the problem (\ref{eq:base}) under the Laplace prior is the LASSO \cite{Tib96}, and there is extensive theory (predicated on knowledge of the true sparsity level) for the properties of this estimator in the $n\ll p$ case \cite{CP07, Wang2009, Wang2011}, as well as the general case. 

While the LASSO has nearly minimax properties in the $n\ll p$ setting (see \cite{DMM09}) when each $\beta_i$ is drawn from a spike-slab model of the form
\[
\beta_i\sim \gamma_i\delta_0(\beta_i)+(1-\gamma_i)\rho(\beta_i)
\]
where $\delta_0$ is the Dirac delta distribution at $0$, $\rho$ is an arbitrary distribution on $\mathbb{R}$, and $\gamma_i\sim\text{Bernoulli}(S/p)$, it has been generally observed that compressible signals exhibit more structure than that encapsulated by the Laplace prior (see \cite{Cev09}). As a result, model-based compressed sensing \cite{BCDH10}, compressible priors \cite{GCD11}, and structured sparsity \cite{HZM09} have been investigated in order to produce estimators with even better properties for the appropriate data sets. 

The methods referenced above incorporate structural heuristics in a way that facilitates computation, which means that they ultimately have to sidestep detailed \emph{a priori} information when it is available. Moreover, from the Bayesian perspective, MAP estimation is only the tip of the iceberg, as it coincides with the Bayes-optimal estimate under the Procrustean loss function ${\bf 1}_{\{\beta^0\}}(\widehat{\beta})$ (which assigns a penalty of zero if $\widehat{\beta}$ is exactly equal to $\beta^0$ and otherwise assigns a penalty of one). The goal of this paper is to establish the existence of a general posterior concentration phenomenon when $n\ll p$, which should encourage the investigation of Bayesian procedures as an alternative.

\subsection{Contributions}

In this paper we demonstrate the existence of a general posterior concentration phenomenon in compressed sensing and the $n\ll p$ sampling regime. For priors concentrated on the set of sparse $\beta$, we show that posteriors concentrate near true sparse $\beta^0$ in the $n\ll p$ regime. Other than this concentration requirement  and standard assumptions from Bayesian analysis, the prior may assume any form.
\begin{itemize}
\item Our main result, Theorem \ref{mainthm}, provides an explicit finite sample bound on the expected concentration of a posterior for an arbitrary prior. The utility of this bound is especially evident when
\begin{enumerate}
\item[(i)] the probability the prior assigns to a small ball around the true $\beta_0$ is not too small;
\item[(ii)] the probability the prior assigns to signals that are not sparse (or approximately sparse) is very small.
\end{enumerate} 
\item We outline the application of Theorem \ref{mainthm} toward producing power calculations for high-probability bounds on model uncertainty, and we discuss the construction of credible regions for the true parameter $\beta^0$ for any choice of prior distribution $\Pi$ on the unknown coefficient vector $\beta$ assuming that (\ref{eq:base}) is the true data-generating model.
\item While Theorem \ref{mainthm} is the first result of its kind that we are aware of, we also show that it is not sharp. We demonstrate that a stronger bound is theoretically possible in at least one difficult, but tractable case (Theorem \ref{UG}). This indicates that the posterior concentration phenomenon is stronger than our current theory suggests. We provide insight into the disparity between these bounds, and discuss the inherent difficulty in obtaining a sharper version of Theorem \ref{mainthm} which applies universally to all priors.
\item Theorem \ref{mainthm} is employed to demonstrate Theorem \ref{asymptopia}, a generic asymptotic posterior contraction result. We then perform more precise analysis for the sparsity($S$)-Gaussian and Bernoulli-Gaussian priors (see Examples \ref{SGexample} and \ref{BGexample} for the definitions of these priors), and determine asymptotic rates of contraction for these types of models. 
\end{itemize}

We work through several examples and obtain asymptotic contraction rates for the sparsity($S$)-Gaussian and Bernoulli-Gaussian priors. The successful contraction in these examples validates our theory for priors which are concentrating very heavily on low-dimensional subspaces of $\mathbb{R}^p$. For priors with small-ball probabilities proportional to the volume of a $p$-dimensional ball, we show that our universal bound is not sufficient to ensure asymptotic contraction. The reason for this is that straightforward lower bounds on the normalization constant for the posterior (the denominator in Bayes's theorem) rely upon Markov's inequality, which suffers from a curse of dimensionality when applied to Gaussian distributions; the lower bounds from Markov's inequality 
\[
1=(2\pi\sigma^2)^{p/2}\int_{\mathbb{R}^p} e^{-\frac{\Vert\beta\Vert_{\lt}^2}{2\sigma^2}}d\beta\geq \frac{e^{-\frac{r_p^2}{2\sigma^2}}}{\sqrt{2\pi\sigma^2}^p}\frac{\pi^{\frac{p}{2}}}{\Gamma(\frac{p}{2}+1)}r_p^p
\]
decay exponentially in $p$ for any sequence of radii $r_p>0$, as well as the optimal sequence $r_p=\sigma\sqrt{p}$. That is to say, the volume of the largest cylinder fitting under the graph of $\mathcal{N}(0,\sigma^2 I_p)$ vanishes as $p\rightarrow\infty$. On the other hand, Markov's inequality is the most convenient choice when we want a universal bound for the normalization constant of the posterior. We essentially employ Markov's inequality twice, once in the proof of Theorem \ref{mainthm} and once again when we estimate small-ball probabilities in our examples. For certain specific instances where a stronger estimate of the normalization constant is available, the implied bound is of course much sharper. For example, we perform such an analysis for the sparsity($S$)-Gaussian prior in Theorem \ref{UG}.

\subsection{Related work}

We draw our main inspiration from the extensive work in compressed sensing. The methods of Cand\`{e}s, Romberg, and Tao \cite{CRT05} began a wave of research into this phenomenon and the precise behavior of estimators. The methods developed along these lines are characterized by geometric bounds and concentration of measure. A similar, more precise (but asymptotic) estimation theory for $\ell_1$-minimization was developed by Donoho and Tanner \cite{DT05}, and rigorous methods related to the replica method have been used to expand greatly upon this theory \cite{BLM12, BM11, DMM09}. However, this asymptotic theory specifically requires the use of random matrices and the lack of convergence rates. Thus, finite sample bounds are not yet accessible from this approach.  

In relation to a more Bayesian approach, the replica method has also been used to obtain results concerning the behavior of MAP estimators in the work of Rangan, Fletcher, and Goyal \cite{RFG12}. Because a rigorous theory for the replica method has not yet been established, making the results of this work rigorous shall require a leap forward in technology.

While there has been some related work in the posterior asymptotic community, the issues that this paper addresses are completely new. Ghosal \cite{Ghosal99} obtained a Bernstein-von Mises theorem providing sufficient conditions for asymptotic normality of the posterior distribution for $\beta$ under model (\ref{eq:base}) allowing non-Gaussian residuals, but the author requires that $p$ grows much slower than $n$. Jiang \cite{Jiang07} studied rates of convergence of the predictive distribution obtained using Bayesian variable selection within a generalized linear model having a diverging number of candidate predictors, but his results focus only on the predictive posterior of $y$ given $X$ and not on the posterior of $\beta$. Bontemps \cite{Bontemps11} obtained a Bernstein-von Mises theorem for a class of semiparametric and nonparametric Gaussian regression models.  For the model (\ref{eq:base}) compared with \cite{Ghosal99}, his results allow a faster growth rate of $p \le n$. However, addressing our interest in $p\gg n$ requires new theory; in this much more challenging case, we do not attempt a Bernstein-von Mises result but instead apply our finite sample probabilistic bounds on the posterior probability assigned to neighborhoods of $\beta_0$ to obtain a general asymptotic contraction result (Theorem \ref{asymptopia}). Using an approach similar to ours,  asymptotic posterior contraction has been studied in the regime $p\le n$ \cite{ADLBS13}.

Along the lines of priors promoting sparsity, a strong theory for the normal means problem has been developed in \cite{CV12} under the assumption that $p=o(n)$. Their asymptotic theory relies upon comparison with a minimax framework. To imitate this theory, the most obvious approach would leverage the framework in \cite{DMM09}, but this would only provide asymptotic guarantees given the current state of that theory. As such, we leave the investigation of this approach to the future.

Another closely related area of research involves the construction of hypothesis tests and confidence intervals based upon the LASSO \cite{JM13a,JM13b,vdGBR13}. These methods are very recent, and we anticipate that these methods may provide a way forward for a sharper analysis of Bayesian model selection. The techniques and principles of those works are quite different from those used in this paper, so we leave this problem as a path of further inquiry.


\subsection{Organization}

In Section 2, we introduce notation and provide background results. Section 3 introduces our main result, the explicit bound on expected posterior concentration for an arbitrary prior and a fixed problem size. We elaborate upon the role of each term in our bound, and discuss how to derive power calculations from the result. Section 3 concludes with an investigation of different directions for sharpening the bound in Theorem \ref{mainthm}. In particular, we demonstrate that for the sparsity($S$)-Gaussian prior,  a sharper bound is attainable using a brute force analysis. In Section 4, we explore asymptotic ramifications of our main inequality. We begin with a general asymptotic theorem, and then proceed to apply our main result to calculate posterior contraction rates for some example priors. Appendices A and B contain the supporting technical material for Section 3.

\section{Preliminaries}

\subsection{Notation}

Assuming that $\beta^0$ is fixed and unknown and that $y$ follows the linear model (\ref{eq:base}), we fix a prior $\Pi$ on $\mathbb{R}^p$ and focus on the posterior for observed data
\begin{equation}
y=X\beta^0+e, \label{model}
\end{equation}
where $y$ is the $n$-dimensional response, $X$ is the $n\times p$ design matrix.\footnote{In compressed sensing literature, this is called the measurement matrix. As we shall see in the next subsection, this matrix must satisfy certain restrictions. In the relevant theory, the matrices with the best provable guarantees are random (see \cite{RV08}), but it is possible to deterministically construct such matrices \cite{BDFKK11,FM11,Iwen2009}. In either scenario, we still refer to $X$ as the design matrix.} In the next subsection, we see that this matrix must satisfy restrictions, 
and the errors are i.i.d. samples from a Normal distribution with mean zero and variance $\sigma^2$ (so $e \sim\mathcal{N}(0,\sigma^2 I_n)$ with $I_n$ the $n$ by $n$ identity matrix).  For a fixed problem $(S,n,p,X,\beta^0)$, we designate the following assumptions:
\begin{enumerate}
\item[(A1)] the $i$th column of $X$ satisfies $\Vert X_i\Vert_{\lt}^2=n$ for all $i=1,\ldots, p$
\item[(A2)] $\beta^0$ is $S$-sparse ($\Vert\beta^0\Vert_{\ell_0}\leq S$)
\item[(A3)] $\sigma$ is known.
\end{enumerate}
Here, $\Vert X_i\Vert_{\lt}^2=\sum_{j=1}^n X_{ji}^2$ and
\[
\Vert\beta^0\Vert_{\ell_0}=\vert\text{supp}(\beta^0)\vert=\vert\{i\in[p]:\beta_i^0\not=0\}\vert.
\]
An $\ell_u$\footnote{In compressed sensing literature, this is typically denoted $\ell_p$. In this paper, we have opted to use the conventions of the statistical literature, and so $p$ is the number of predictors.} ball of radius $\varepsilon$ centered at $\beta$ is denoted
\[
B_\varepsilon^{\ell_u}(\beta)=\left\{x\in\mathbb{R}^p:\Vert x-\beta\Vert_{\ell_u}=\left(\sum_{i=1}^p\vert x_i-\beta_i\vert^u\right)^{1/u}<\varepsilon\right\}.
\]

\begin{defn}
For any $\beta\in\mathbb{R}^p$ and any natural number $S\leq p$, let $\sigma_S(\beta)$ denote the best $S$-term approximation error of $\beta$ so that
\begin{equation}
\sigma_S(\beta)=\inf_{\Vert b\Vert_{\ell_0}\leq S}\Vert\beta-b\Vert_\lw.\footnote{Approximation error in the $\lw$ norm is chosen for simplicity, but our results hold for other norms as well.}
\end{equation}
Furthermore, for any $R\geq 0$, let
\begin{equation} 
\mathcal{P}_{S,R}=\{\beta\in\mathbb{R}^p:\sigma_S(\beta)\leq R\}\label{sieve}
\end{equation}
denote the set of $(S,R)$-compressible vectors.
\end{defn}
\noindent Note that $\mathcal{P}_{S,0}$ is exactly the union of canonical $S$-dimensional subspaces in $\mathbb{R}^p$. When $S$ and $R$ are clear from the context, we shall simply let $\mathcal{P}=\mathcal{P}_{S,R}$.

Given the model (\ref{model}), we let $\mathcal{L}(\beta\vert y):=f(y\vert\beta)$ denote the likelihood of $\beta\in\mathbb{R}^p$ given the outcome $y\in\mathbb{R}^n$, and hence
\begin{equation}
f(y\vert\beta)=(2\pi\sigma^2)^{-n/2}\exp\{-\Vert y-X\beta\Vert_{\lt}^2/2\sigma^2\}.
\end{equation}
For any $\beta\in\mathbb{R}^p$, Borel measurable $U\subset\mathbb{R}^n$, and Borel measurable function $F:\mathbb{R}^n\rightarrow\mathbb{R}$, we let
\begin{equation}
\text{pr}_\beta(U)\text{ and }\mathbb{E}_\beta F
\end{equation}
denote the probability of the event $y\in U$ given the parameter $\beta$ and the expectation of $F(y)$ given the parameter $\beta$, respectively. We also let $U^c$ denote the complement of the set $U$, and we let ${\bf 1}_U$ denote the indicator function of $U$.
Finally, for a linear operator $X: \mathbb{R}^p \rightarrow \mathbb{R}^n$, we use the notation $\|X\|_{\ell_u \rightarrow \ell_v}$ to denote the operator norm 
\begin{equation}\|X\|_{\ell_u \rightarrow \ell_v} = \max_{\{\beta: \|\beta\|_{\ell_u}=1\}}\|X\beta\|_{\ell_v}.\end{equation}
\subsection{The Dantzig selector}
The Dantzig selector  is an important ingredient for our proofs. 
The properties of the Dantzig selector depend upon the design matrix $X$, and one simple assumption laid out by Cand\`{e}s and Tao \cite{CT05} is that the column norms of $X$ all equal one. Because our noise model has expected magnitude $\sigma\sqrt{n}$, we rescale $X$ so the column norms are all $\sqrt{n}$. 
We shall use $\widetilde{X}=\frac{1}{\sqrt{n}}X$ as an intermediate quantity to translate the results of Cand\`{e}s and Tao to our setting. 

\begin{defn}
For a response vector $y\in\mathbb{R}^n$ and a design matrix $X$, the Dantzig selector is the solution to the program
\begin{equation}
\min\Vert\beta\Vert_\lw\text{ subject to }\Vert \widetilde{X}^T(y-\widetilde{X}\beta)\Vert_\linf\leq\lambda_p\sigma\label{dantzigprog}
\end{equation}
where $\lambda_p=\sqrt{2(1+\alpha)\log p}$. The role of the free parameter $\alpha>0$ is made apparent in Theorem \ref{dantzigthm}. We let $\widetilde{\beta}$ denote the solution to this linear programming problem, and set $\widehat{\beta}=\widetilde{\beta}/\sqrt{n}$ (we rescale because the theory for the Dantzig selector is for design matrices with unit column norms).
\end{defn}

In order to ensure reconstruction properties for the Dantzig selector for all $\beta$ of a sufficient sparsity, we must put conditions on $\widetilde{X}$. The first quantity of interest is the restricted isometry constant, which is the smallest constant $\delta_k(\widetilde{X})$ satisfying
\begin{equation}
(1-\delta_k)\Vert b\Vert_{\lt}^2\leq\Vert \widetilde{X}b\Vert_{\lt}^2\leq(1+\delta_k)\Vert b\Vert_{\lt}^2
\end{equation}
for all $b\in\mathcal{P}_{k,0}$. Ideally, the constant $\delta_k$ is small enough to ensure that sufficiently sparse $b$ are far from the kernel of $\widetilde{X}$. The other quantity of interest is the restricted orthogonality constant $\theta_{k,k^\prime}(\widetilde{X})$, which is defined to be the smallest constant such that 
\begin{equation}
\vert\langle \widetilde{X}_Tb,\widetilde{X}_{T^\prime}b^\prime\rangle\vert\leq\theta_{k,k^\prime}\Vert b\Vert_{\lt}\Vert b^\prime\Vert_{\lt}
\end{equation}
for all $b$, $b^\prime$, disjoint $T,T^\prime\subset \{1,2,\ldots,p\}$, $\vert T\vert\leq k$, $\vert T^\prime\vert\leq k^\prime$, and $|T|+|T^\prime|\leq p$. Here, $\widetilde{X}_T$ and $\widetilde{X}_{T^\prime}$ are the matrices formed by respectively concatenating the columns of $\widetilde{X}$ with indices in $T$ and $T^\prime$. Again, the ideal $\theta_{k,k^\prime}$ is small, so disjoint collections of columns of $\widetilde{X}$ span nearly orthogonal subspaces.

With the restricted isometry and restricted orthogonality constants defined, we are now able to translate the theorem of Cand\`{e}s and Tao into our setting.

\begin{thm}[Cand\`{e}s and Tao '05]\label{dantzigthm}
Let $S$ be fixed so that $\delta_{2S}(\widetilde{X})+\theta_{S,2S}(\widetilde{X})<1$ and fix $R\geq 0$. If $\beta^0\in\mathcal{P}_{S,R}$, then the rescaled solution to (\ref{dantzigprog}) satisfies
\begin{equation}
\Vert\widehat{\beta}-\beta^0\Vert_\lt\leq\frac{4\sqrt{2}\sigma}{1-\delta-\theta}\sqrt{\frac{2(1+\alpha)S\log p}{n}}+2\frac{1-\delta+\theta}{1-\delta-\theta}\frac{R}{\sqrt{S}}
\end{equation}
with probability greater than $1-\frac{1}{p^\alpha\sqrt{\pi\log p}}$.
\end{thm}

For completeness, we prove this version of the theorem in Appendix A. Since we shall employ the above condition on $\widetilde{X}$ to invoke this theorem and to perform further analysis, we add the following assumption.
\begin{enumerate}
\item[(A4)] $\delta\equiv\delta_{2S}(\widetilde{X})$ and $\theta\equiv\theta_{S,2S}(\widetilde{X})$ satisfy $\delta+\theta<1$
\end{enumerate}
In the case of increasing problem sizes, we shall assume that $\delta$ and $\theta$ remain fixed (or are at least nonincreasing as the problem size increases). While at first glance this may seem to constrain the applicability of our theory, such conditions are standard in the theoretical literature on sparse reconstructions and obtaining universal statements in problems where $n\ll p$ without similar conditions is an open problem. 

Another possible concern is that verification of these constants is combinatorially complex, however it has been well established that many families of random matrices satisfy this condition with high probability when $n \geq CS \log^c(p/S)$. In particular, matrices whose entries are drawn i.i.d. $\mathcal{N}(0,1/n)$, and matrices with sub-Gaussian entries satisfy this condition with high probability with $c=1$. 
More specifically, if $\sqrt{n}X$ is an $n\times p$ random matrix with independent, isotropic, and sub-Gaussian%
 \footnote{A random variable $Y$ is sub-Gaussian if $\Pr(|Y|\geq t) < c_1 e^{-c_2 t^2}$ for some $c_1,c_2>0$ and all $t>0$. A random vector $Z$ is sub-Gaussian if $ \langle Z, x\rangle$ is sub-Gaussian for all $x$ with $\|x\|_2=1$ (and constants $c_1$ and $c_2$ independent of $x$). A random vector $Z$ is isotropic if $\mathbb{E}|\langle Z,x \rangle|^2 = \|x\|_2^2$ for all $x$. See \cite{FoucartRauhut2013}[Chapter 9].}
 rows with $n\geq C\hat\delta^{-2}(S\ln(ep/S) +\ln(2\varepsilon^{-1}) )$, then the restricted isometry constant of $X$ satisfies $\delta_S \leq \hat\delta$, with probability at least $1-\varepsilon$ (for example $C\approx 80.098$ for Gaussian matrices), see \cite{FoucartRauhut2013}. Noting that $\theta_{S,2S}\leq \delta_{3S}$, this gives an idea of the size of $\delta$ and $\theta$, for such sub-Gaussian random matrices, as a function of the dimensions. 
Other matrices that satisfy such restricted isometry conditions (albeit with $c>1$) include $n\times p$ matrices whose rows are drawn (uniformly) at random from orthonormal bases such as the discrete Fourier basis. The interested reader is referred, for example, to \cite{CRT05, FoucartRauhut2013}.
\section{The finite sample bound and power calculations}

In this section, we state and describe our universal posterior concentration bound for finite sample sizes. This bound gives us a precise power calculation whenever we desire to bound uncertainty in our model. Moreover, it may be specialized to provide asymptotic analyses for various priors, as shall be demonstrated in the Section 5. The proof of Theorem \ref{mainthm} is carried out in Appendix A.

\begin{thm}\label{mainthm}
Suppose $\beta^0\in\mathcal{P}=\mathcal{P}_{S,R}$ and that $\Pi$ is an arbitrary prior on $\mathbb{R}^p$.  Let
\begin{eqnarray}
\varepsilon = \frac{8\sigma}{1-\delta-\theta}\sqrt{\frac{(1+\alpha)S\log p}{n}}+2\frac{1-\delta+\theta}{1-\delta-\theta}\frac{R}{\sqrt{S}},
\end{eqnarray}
and assume (A1), (A3) and (A4). For any $\alpha>0$ and $\tau>0$ with $0<1+\tau<\alpha$, we have
\begin{eqnarray}
\mathbb{E}_{\beta^0}\Pi(B_{2\varepsilon}^{\lt}(\beta^0)^c\vert y)&\leq& \frac{1}{p^\alpha\sqrt{\pi\log p}}\label{fixedComp}\\
&&+\frac{p^{1+\tau}\Pi(B_{2\varepsilon}^{\lt}(\beta^0)^c\setminus\mathcal{P})}{\Pi(B_{C_\tau}^{\lw}(\beta^0))}\label{sparseComp}\\
&&+\frac{p^{1+\tau}}{p^\alpha\sqrt{\pi\log p} \Pi(B_{C_\tau}^{\lw}(\beta^0))}\label{concentrationComp}\\
&&+\frac{1}{p^{\tau}\sqrt{\pi\log p}}\label{matrixComp},
\end{eqnarray} 
where
\begin{equation}
C_\tau=\frac{1}{3}\sqrt{\frac{2\sigma^2(1+\tau)\log p}{n}}.
\end{equation}
\end{thm}

We now describe the contributions of each of the terms in the inequality in Theorem \ref{mainthm}, noting first that the terms (\ref{fixedComp}) and (\ref{matrixComp}) are independent of $\Pi$. The term in (\ref{fixedComp}) comes from using the Dantzig estimator in our hypothesis test within the proof. As such, we have little control over this term aside from adjusting the parameter $\alpha$.  The fourth term, (\ref{matrixComp}) is controlled by the parameter $\tau$, and arises when we put a bound on the noise. When $p$ is very large, the net contribution of these terms is insignificant.

Having discussed the terms that are independent of the prior, we turn our attention to the middle terms. The term (\ref{concentrationComp}) depends inversely upon $\Pi(B_{C_\tau}^{\lw}(\beta^0))$, the probability the prior assigns to a small ball around $\beta^0$. The behavior of this term illustrates the role that sparsity plays in the behavior of the posterior. In order to control this term, we must increase $\alpha$. However, if $\Pi(B_{C_\tau}^{\lw}(\beta^0))$ is proportional to the volume of $B_{C_\tau}^{\lw}(\beta^0)\subset\mathbb{R}^p$, then $\alpha$ must overcome $p$, and
\[
\varepsilon\propto \sqrt{\frac{S(1+Cp)\log p}{n}},
\]
may be quite large. This would mean that asymptotic contraction is not feasible. On the other hand, a sparsity promoting prior can lead to posterior concentration. Because a sparsity promoting prior is concentrated very near $S$-dimensional subspaces, the probability assigned to a small ball around a sparse vector is proportional to the volume of a ball in $\mathbb{R}^S$. Thus, $\alpha$ can remain $O(S)$, and $\varepsilon$ shrinks asymptotically if $S^2\log p$ is $o(n)$.

Finally, we discuss the term (\ref{sparseComp}). It is clear that this term can only be controlled if the prior encourages sparse $\beta$.  In particular, if $\Pi$ is a compressible prior (see \cite{Cev09, GCD11}), $\Pi(\mathcal{B}\setminus\mathcal{P})$ should be small. In general, the decay of $\Pi(B_{2\varepsilon}^{\lt}(\beta^0)^c\setminus\mathcal{P})$ must overcome the growth of a $p^S$ term produced by $\Pi(B_{C_\tau}^{\lw}(\beta^0))$. 

\subsection{Statistical power calculations for model uncertainty}

For a statistical power calculation of a credible region, a statistician specifies
\begin{enumerate}
\item a $\xi\in(0,1)$ that controls the significance of the region,
\item a $\rho\in(0,1)$ such that the posterior concentration on the region is at least $1-\rho$ with probability at least $1-\xi$ on the draw of $y$,
\item an $r>0$ such that the radius of the region is at most $r$,
\end{enumerate}
and then proceeds to compute the minimal sample size $n$ such that the above constraints are all satisfied. In the $n\ll p$ scenario, we must relax the constraint on $r$, and minimize a cost function $Q(r,n)$ over the set of all $(r,n)$ such that the posterior concentration on the ball of radius $r$ is at least $1-\rho$ with probability $1-\xi$. The function $Q$ is any function such that $Q(r,n)\leq Q(r^\prime, n^\prime)$ if $r\leq r^\prime$ and $n\leq n^\prime$.

Applying Markov's inequality to the bound in Theorem \ref{mainthm}, we obtain
\begin{eqnarray*}
\rho\cdot pr_{\beta^0}\{\Pi(B_r^{\lt}(\beta^0)^c\vert y)>\rho\}&\leq& \frac{1}{p^\alpha\sqrt{\pi\log p}}\left(1+\frac{p^{1+\tau}}{\Pi(B_{C_\tau}^{\lw}(\beta^0)^c)}\right)\\
&&+\frac{p^{1+\tau}\Pi(B_r^{\lt}(\beta^0)^c\setminus\mathcal{P})}{\Pi(P_{C_\tau}^{\lw}(\beta^0))}\\
&&+\frac{1}{p^{\tau}\sqrt{\pi\log p}}\\
\end{eqnarray*}
for
\begin{eqnarray*}
r=2\left(\frac{8\sigma}{1-\delta-\theta}\sqrt{\frac{(1+\alpha)S\log p}{n}}+2\frac{1-\delta+\theta}{1-\delta-\theta}\frac{R}{\sqrt{S}}\right).
\end{eqnarray*}
Since $p^{-\alpha}$ and $\Pi(B_r^{\lt}(\beta^0)^c\setminus\mathcal{P})$ can be made arbitrarily small for a large enough choice of $\alpha$, for any fixed $n$ and $\tau$ such that
\[
\frac{\rho^{-1}}{p^\tau\sqrt{\pi\log p}}<\xi
\]
there is an $\alpha$ (and hence an $r$) such that $pr_\beta\{\Pi(B_r^{\lt}(\beta^0)^c)>\rho\}<\xi$. Once we have estimates for $\Pi(B_r^{\lt}(\beta^0)^c\setminus\mathcal{P})$ and $\Pi(B_{C_\tau}^{\lw}(\beta^0))$ in terms of $\alpha$ and $\tau$, this indicates that we may optimize and get the smallest $\alpha$ (and hence the smallest $r$). After we have fixed $\alpha$ and $n$ subject to the above constraints, the posterior concentration on $B_{4\varepsilon}^{\lt}(\widehat{\beta})$ is then greater than $1-\rho$ with probability exceeding $1-\xi-\frac{1}{p^\alpha\sqrt{\pi\log p}}$.

Just as in the theory of compressed sensing, Theorem \ref{mainthm} requires prior information about the sparsity level $S$. Unlike compressed sensing, Theorem \ref{mainthm} also requires an estimate for the prior probability on the ball $B_{C_\tau}^{\lw}(\beta^0)$ and on the set $B_{2\varepsilon}^{\lt}(\beta^0)^c\setminus\mathcal{P}$ for all $\beta^0$ of interest. Such bounds are generally available if we have bounds (for absolutely continuous $\Pi$ with density function $\Pi(\beta)$) on $\Pi(\beta^0)$ since we may always use the estimate
\[
\inf_{\beta\in U} \Pi(\beta)\text{Vol}(U)\leq \Pi(U)\leq \sup_{\beta\in U} \Pi(\beta)\text{Vol}(U)
\]
for any measurable $U$, where $\text{Vol}(U)$ is the Lebesgue measure of the set $U$. Since we must observe several examples before we can conclude the sparsity of the class of examples, it is simple to simultaneously determine the set of $\beta$ to determine bounds on $\Pi(\beta)$ over the entire class.

\subsection{Sharpening Theorem \ref{mainthm}}

The proof of Theorem \ref{mainthm} employs numerous inequalities, and it is instructive to determine which are tight. The most important estimate to examine is the lower bound of the normalization constant (the denominator in Bayes's Theorem). The proof of Theorem \ref{mainthm} employs Markov's inequality to bound the expression
\[
\int \exp\left\{-\frac{1}{2\sigma^2}\left [\Vert X\beta-y\Vert_{\lt}^2-\Vert X\beta^0-y\Vert_{\lt}^2\right]\right\}d\Pi(\beta)
\]
through the reduction
\begin{eqnarray*}
\Vert X\beta-y\Vert_{\lt}^2-\Vert X\beta^0-y\Vert_{\lt}^2&=&\langle X(\beta-\beta^0),X(\beta-\beta^0)\rangle -2\langle X(\beta-\beta^0),y-X\beta^0\rangle\\
&=&\langle X(\beta-\beta^0),X(\beta-\beta^0)\rangle -2\langle X(\beta-\beta^0),e\rangle.
\end{eqnarray*}
Because $n<p$, $X$ has a nontrivial kernel, $X(\beta-\beta^0)=XQ(\beta-\beta^0)$ where $Q$ is the orthogonal projection onto the cokernel of $X$. By replacing $\beta-\beta^0$ with $Q(\beta-\beta^0)$ in the equations (4.54) to (4.58), we see that we may replace $\Pi(B_{C_\tau}^{\lw}(\beta^0))$ with $\Pi(\mathcal{C})$ where
\[
\mathcal{C}=\{\beta\in\mathbb{R}^p:\Vert Q(\beta-\beta^0)\Vert_{\lw}\leq C_\tau\}.
\]
This set is essentially cylindrical, and represents a much larger proportion of the domain than the metric ball. Moreover, the prior probability on this set is proportional to the volume on an $n$-dimensional metric ball. The dependence of $\mathcal{C}$ on $X$ brings additional complexities to the concentration calculations, so it is generally more straightforward to compute using $\Pi(B_{C_\tau}^{\lw}(\beta^0))$. This is especially true when we consider asymptotic calculations.

Now, the useful design matrices that we have examined generally have nonzero singular values which are $O(\sqrt{p})$. This means that, in high dimensions the likelihood term  $e^{-\frac{1}{2\sigma^2}\Vert X\beta-y\Vert_{\lt}^2}$ is highly concentrated around the affine subspace $\{\beta\in\mathbb{R}^p:X\beta=y\}$. Consequently, the concentration behavior for an absolutely continuous prior $\Pi$ is generally dictated by the singular distribution
\[
\frac{\Pi(\beta)\delta_{X\beta=y}(\beta)}{\displaystyle\int \Pi(\beta)\delta_{X\beta=y}(\beta)d\beta}
\]
where $\Pi(\beta)$ is the Radon-Nikodym derivative of $\Pi$ with respect to Lebesgue measure, and $\delta_{X\beta=y}$ is the pushforward measure\footnote{If $f$ is a measurable map between two measure spaces $\mathcal{X}$ and $\mathcal{Y}$, and $\mu$ is a probability measure on $\mathcal{X}$, the {\it pushforward measure} of $\mu$ under $f$, $\mu_f$, is the measure satisfying $\mu_f(U)=\mu(f^{-1}(U))$ for all measurable $U\subset\mathcal{Y}$.} of $(p-n)$-dimensional Lebesgue measure under isometric identification with the affine subspace $\{\beta\in\mathbb{R}^p: X\beta=y\}$. With tight characterizations of these distributions for general $X$ and $y$, it is in principle possible to obtain nearly optimal concentration bounds. The difficulty with this ideal is that the collection of these distributions for an arbitrary prior $\Pi$ introduces additional hypotheses on $\Pi$ which are complicated to verify. From this perspective, Theorem \ref{mainthm} may be seen as an intermediate result which trades complexity for precision, and Theorem \ref{mainthm} is still able to produce reasonable bounds for priors that concentrate on very low dimensional subspaces of $\mathbb{R}^p$. 

Given a sharp estimate for the normalization constant of the posterior, we can obtain much better bounds than are available via Markov's inequality. As an example of the precision that can be obtained in special cases, we consider a brute-force analysis of the concentration for the sparsity($S$)-Gaussian prior.


\begin{thm}\label{UG}
Assume (A1) through (A4) and that $\Pi$ is the sparsity($S$)-Gaussian prior with parameters $S$ and $V$. Further assume that $\Vert\beta^0\Vert_{\ell_\infty}\leq C<\infty$. Fix $\alpha>0$, suppose $\delta < 29/31$,  
and let
\begin{eqnarray}
\varepsilon&=&\frac{C\sigma^2\sqrt{S}/n+(C_2\sigma V^2+C_3\sigma^2/n)\sqrt\frac{(1+\alpha)S\log p}{n}}{(1-\delta)V^2+\sigma^2/n},
\end{eqnarray}
where the positive constants $C_2$ and $C_3$ depend only $\delta$ and $\theta$.
If $(n\frac{1-\delta}{\sigma^2}+\frac{1}{V^2})\varepsilon^2\geq S/2$, then there exists a constant $\eta = \eta(\alpha, \delta, \theta)>0$ so that
\begin{equation}
\Pi(B_{2\varepsilon}^{\lt}(\beta^0)\vert y)\geq\frac{1-e^{-\frac{1}{4}(n\frac{1-\delta}{\sigma^2}+\frac{1}{V^2})\varepsilon^2}}{1+\left(e^2\frac{n(1+\delta)V^2+\sigma^2}{n(1-\delta)V^2+\sigma^2}\right)^{S/2}e^{\frac{1}{n(1-\delta)V^2+\sigma^2}\frac{\Vert y\Vert_{\lt}^2}{2}}S^{-S}p^{-\eta S}}
\end{equation}
with probability greater than $1-1/p^{\alpha}\sqrt{\pi\log p}$ on the draw of $y$. 
\end{thm}

First, we note that we must invoke the additional assumption 
$$
\Vert\beta^0\Vert_{\ell_\infty}\leq C<\infty.
$$
While this is not a required assumption for the fidelity of compressed sensing point estimates, assumptions of this form arise when we are asking about the global behavior of the posterior. Finite sample concentration bounds for any posterior fundamentally depend upon the magnitude of $\beta^0$ because the concentration of the prior decays as the magnitude of $\beta^0$ increases. In any case, we may still estimate this quantity in a practical setting and the parameters $V$ and $n$ can be increased to ameliorate the effect of $C$.

Now, comparing this bound with the one that we obtain in Example \ref{SGexample} below, it is clear this theorem is sharper. In particular, note that we no longer need to scale $\alpha$ to obtain asymptotic contraction. A very crude approximation in the asymptotic regime would be $\varepsilon\approx \sqrt{\frac{S\log p}{n}}$ and
\begin{equation}
\Pi(B_{2\varepsilon}^{\lt}(\beta^0)\vert y)\approx \frac{1- Q_1p^{-\eta_1 S}}{1+Q_2p^{-\eta_2 S}}
\end{equation}
with probability exceeding $1-1/p^\alpha\sqrt{\pi\log p}$. In order to obtain contraction, we simply let $S\log p=o(n)$. This is as good a result as one might hope for, as (depending on $X$) $n$ must be at least $C S\log(p/S)$ to guarantee (A4) and the error for the compressed sensing estimators is generally controlled by $C \frac{S}{n}\log p$.  

\section{Asymptotic applications}

Based on Theorem \ref{mainthm}, we may exhibit a general asymptotic posterior contraction result depending upon $\Pi(B_{2\varepsilon}^{\lt}(\beta^0)^c\setminus\mathcal{P})$ and $\Pi(B_\tau^{\lw}(\beta^0))$. Whenever we consider a sequence of problems, $(S_n,n,p_n,X(n),\beta^0(n))$ with $n\rightarrow\infty$, we additionally assume that $\vert\beta_i^0(n)\vert\leq C<\infty$ for all $i$ and $n$ and that the constants $C$ and $\sigma$ remain fixed as $n$ increases. A certain level of control of $\Vert\beta^0(n)\Vert_{\lt}$ is an essential ingredient in any asymptotic analysis since it is a surrogate for the prior concentration around $\beta^0$, and contraction is an impossibility if this prior concentration is shrinking too quickly. We briefly recall the definition of posterior consistency to motivate Theorem \ref{asymptopia}. For a more extensive treatment of consistency, the interested reader is referred to the treatise of Ghosh and Ramamoorthi \cite{GR03}.

\begin{defn}
A sequence of posteriors $\Pi_n$ is said to be consistent for the sequence $(n,p_n,X(n),\beta^0(n))$ if 
\[
\Pi_n(B_r^{\lt}(\beta^0(n))\vert y_n)\longrightarrow 1
\]
almost surely on the draw of the sequence $y_n=X(n)\beta^0(n)+e_n$ for all fixed $r>0$.
\end{defn}

\begin{thm}\label{asymptopia}
Suppose $(S_n,n,p_n,X(n),\beta^0(n))$ is a sequence of problems satisfying (A1) through (A4), and that $\Pi_n$ is a sequence of priors on $\mathbb{R}^{p_n}$ such that
\begin{itemize}
\item[i.] $\Pi(\mathcal{D}_n)\geq p_n^{-\eta_n}$ 
\item[ii.] $\Pi(\mathcal{B}_n\setminus\mathcal{P})\leq p_n^{-\phi_n}$
\end{itemize}
for sequences of positive constants $\{\eta_n\}_{n\geq 1}$, $\{\phi_n\}_{n\geq 1}$, and where
\begin{equation}
\mathcal{D}_n=\left\{\beta\in\mathbb{R}^{p_n}:\Vert\beta-\beta^0(n)\Vert_{\lw}<\frac{1}{3}\sqrt{\frac{2\sigma^2(1+\alpha_n/2)\log p_n}{n}}\right\}.
\end{equation}
Let $\mathcal{B}_n=\{\beta\in\mathbb{R}^p:\Vert\beta-\beta^0(n)\Vert_\lt>2\varepsilon_n\}$ where
\begin{eqnarray}
\varepsilon_n = \frac{8\sigma}{1-\delta-\theta}\sqrt{\frac{(1+\alpha_n)S_n\log p_n}{n}}.
\end{eqnarray}
Then 
\begin{equation}
\mathbb{E}_{\beta^0}\Pi_n(\mathcal{B}_n\vert y)\leq p_n^{1+\alpha_n/2+\eta_n-\phi_n}+\left(p^{-\alpha_n/2}+p^{1+\eta_n}+1\right)\frac{1}{p^{\alpha_n/2}\sqrt{\pi\log p}},
\end{equation}
and hence we have posterior consistency if
\[
\alpha_n\geq 1+q,\: \alpha_n-2\eta_n-2\geq q,\text{ and }2\phi_n-\alpha_n-2\eta_n-2\geq q
\]
asymptotically for some fixed $q>0$.
\end{thm}

We now examine a couple of cases where we may perform more precise calculations to obtain rates of asymptotic contraction.

\begin{exmp}\label{SGexample}
First, we turn our attention to a case that admits the simplest (but still somewhat involved) analysis. This is the case of the sparsity($S$)-Gaussian prior. We let
$\{0,1\}_S^p$ denote the $p$-length binary sequences with exactly $S$ nonzero entries and fix the model
\begin{eqnarray}
\beta_i\sim \gamma_i\mathcal{N}(0,V^2)+(1-\gamma_i)\delta_0\\
\gamma\sim\text{Uniform}(\{0,1\}_S^p)
\end{eqnarray}
where $\text{Uniform}(\{0,1\}_S^p)$ is the distribution with equal ($1/\binom{p}{S}$) probability for all $\gamma\in\{0,1\}_S^p$. First, note that $\Pi(B_{2\varepsilon}^{\lt}(\beta^0)^c\setminus\mathcal{P}_S)=0$. This eliminates the term (\ref{sparseComp}) in Theorem \ref{mainthm} and we obtain
\begin{equation}
\mathbb{E}_{\beta^0}\Pi(B_{2\varepsilon}^{\lt}(\beta^0)\vert y)\leq\left(2+\frac{p^2}{\Pi(B_{C_1}^{\lw}(\beta^0))}\right)\frac{1}{p^\alpha\sqrt{\pi\log p}}\label{red1:1}
\end{equation}
when we set $\tau=1$. We now only need to estimate $\Pi(B_{C_\alpha}^{\lw}(\beta^0))$. To that end, suppose that $\beta^0$ has support $T$ \footnote{Here, we assume that $\vert T\vert = S$, but note that our following bounds are equally valid for $\beta^0$ with sparsity at most $S$.} and denote by $\text{Vol}(\mathcal{D})$ the volume of an $S$-dimensional $\lw$-ball with radius $C_1$. Let $M$ denote the minimum of $\prod_{i\in T}\mathcal{N}(\beta_i\vert 0,V^2)$ over $\mathcal{D}$.  Then
\begin{eqnarray}
\Pi(B_{C_1}^{\lw}(\beta^0))&\geq& M\text{Vol}(\mathcal{D})\Pi(\gamma={\bf 1}_T)=\binom{p}{S}^{-1}M\text{Vol}(\mathcal{D})\\
&\geq&\binom{p}{S}^{-1} (2\pi V^2)^{-S/2}e^{-\Vert\beta^0\Vert_\lt^2/2V^2} e^{-C_1^2/2V^2}\frac{(2C_1)^S}{\Gamma(1+S)}\\
&\geq&\frac{e^{-\Vert\beta^0\Vert_\lt^2/2V^2-C_1^2/2V^2}}{\sqrt{2\pi (eV)^2}^S}\left(\frac{2C_1}{p}\right)^S\label{red1:2}
\end{eqnarray}
and substitution for $C_1$ yields
\begin{eqnarray}
\Pi(B_{C_1}^{\lw}(\beta^0))&\geq&e^{-C_1^2/2V^2}\left(\frac{2\sigma}{3e^{C^2/2V^2}\sqrt{\pi (eV)^2}}\right)^S\left(\frac{1}{p}\sqrt{\frac{\log p}{n}}\right)^S\nonumber\\
&=&\eta_0(\eta_1)^S\left(\frac{1}{p}\sqrt{\frac{\log p}{n}}\right)^S\label{red1:4}
\end{eqnarray}
where we have set $\eta_0=e^{-C_{\nu,\kappa}^2/2V^2}$ and $\eta_1=e^{-C^2/2V^2}\sqrt{\frac{2\sigma^2}{\pi e^2 V^2}}$.
Note that $\eta_1$ is constant, and (though it  depends on $n, p,$ and $\alpha$) $\eta_0$ is approximately constant in the asymptotic regime. Combining (\ref{red1:4}) with (\ref{red1:1}), we have the bound
\begin{equation}
\mathbb{E}_{\beta^0}\Pi(\mathcal{B}\vert y)\leq\left(2+\frac{p^2}{\eta_0}\left(\frac{p}{\eta_1}\sqrt{\frac{n}{\log p}}\right)^S\right)\frac{1}{p^\alpha\sqrt{\pi\log p}}.\label{red1:5}
\end{equation}

Now, consider a sequence of problems $(S_n,n,p_n,X(n),\beta^0(n))$ satisfying (A1) through (A4), and suppose we employ the sparsity($S$)-Gaussian prior with parameter $V$ fixed for each $n$. Then the bound in (\ref{red1:5}) applies to $\mathbb{E}_{\beta^0(n)}\Pi(\mathcal{B}_n\vert y(n))$ for each $n$, where the radius of $\mathcal{B}_n$ is $2\varepsilon_n$ with
\begin{eqnarray}
\varepsilon_n = \frac{8\sigma}{1-\delta-\theta}\sqrt{\frac{(1+\alpha_n)S_n\log p_n}{n}}.
\end{eqnarray}
The most problematic contribution to the bound in $(\ref{red1:5})$ is $p_n^{S_n}$, but we may adjust $\alpha_n$ so that $p_n^{\alpha_n}$ overcomes this term asymptotically.  Thus, in order to obtain asymptotic consistency, we require $\alpha_n-S_n\rightarrow\infty$ and $(1+\alpha_n)S_n\log p_n=o(n)$. This is possible if we set $\alpha_n=S_n\log p_n$ and assume $S_n\log p_n=o(\sqrt{n})$. Comparing this with the rate implied by Theorem \ref{UG}, we see that the sharper analysis gives us $S_n\log p_n = o(n)$.
\end{exmp}

\begin{exmp}\label{BGexample}
Now, we examine the case where $\Pi$ follows the Bernoulli-Gaussian model,
\begin{eqnarray}
\beta_i\sim \gamma_i\mathcal{N}(0,V^2)+(1-\gamma_i)\delta_0\\
\gamma_i\sim\text{Bernoulli}(\phi)
\end{eqnarray}
where $\phi\in(0,1)$ controls the sparsity of the prior. We assume that $\beta^0$ is $K$-sparse and that $p\phi=K$. By Chernoff-Hoeffding, we have that
\begin{equation}
\text{pr}\left\{\sum\gamma_i\geq S\right\}\leq \left(\frac{K}{S}\right)^{S}\left(\frac{p-K}{p-S}\right)^{p-S}.
\end{equation}
Note that this is a bound for $\Pi(B_{2\varepsilon}^{\lt}(\beta^0)^c\setminus\mathcal{P}_{S,0})$. We are left with producing an estimate for $\Pi(B_{C_1}^{\lw}(\beta^0))$:
\begin{eqnarray}
\Pi(B_{C_1}^{\lw}(\beta^0))&=& \sum_{\gamma}\Pi(B_{C_1}^{\lw}(\beta^0)\vert\gamma)\Pi(\gamma)\\
&=&\sum_{k=0}^{p-K}\binom{p-K}{k}\phi^{K+k}(1-\phi)^{p-K-k}\Pi(B_{C_1}^{\lw}(\beta^0)\vert\gamma)\\
&\geq&\phi^K\sum_{k=0}^{p-K}\binom{p-K}{k}\phi^k(1-\phi)^{p-K-k}\frac{e^{-\frac{\Vert\beta^0\Vert_\lt^2}{2V^2}-\frac{C_1^2}{2V^2}}}{\sqrt{2\pi V^2}^{K+k}}\frac{(2C_1)^{K+k}}{\Gamma(1+K+k)}\nonumber\\
&\geq&\eta_0\left(\frac{2\phi C_1}{e^{C^2/2V^2}\sqrt{2\pi V^2}}\right)^K\sum_{k=0}^{p-K}\binom{p-K}{k}\left(\frac{2\phi C_1}{\sqrt{2\pi V^2}}\right)^k(1-\phi)^{p-K-k}\frac{1}{(K+k)!}\nonumber\\
&\geq&\eta_0\left(\frac{2\phi C_1}{K\sqrt{2\pi V^2}}\right)^K\sum_{k=0}^{p-K}\binom{p-K}{k}\left(\frac{2\phi C_1}{p\sqrt{2\pi V^2}}\right)^k(1-\phi)^{p-K-k}\nonumber\\
&=&\eta_0\left(\frac{\eta_1}{p}\frac{S}{K}\sqrt{\frac{\log p}{n}}\right)^K\left(1-\frac{K}{p}+\eta_1\frac{K}{p^2}\sqrt{\frac{\log p}{n}}\right)^{p-K}.
\end{eqnarray}
Here, $\eta_0$ and $\eta_1$ are similar to their counterparts in the previous example. In this case, the term (\ref{sparseComp}) in Theorem \ref{mainthm} is bounded by
\begin{equation}
\frac{p}{\eta_0}\left(\frac{K}{S}\right)^{S+K}\left(\frac{p}{\eta_1}\sqrt{\frac{\log p}{n}}\right)^K\left(\frac{p-K}{p-S}\right)^{p-S}\left(1-\frac{K}{p}+\eta_1\frac{K}{p^2}\sqrt{\frac{\log p}{n}}\right)^{K-p}.
\end{equation}
\normalsize 
Now, consider a sequence of problems $(K_n,n,p_n,X(n),\beta^0(n))$ satisfying (A1) through (A4), and suppose we employ the Bernoulli-Gaussian prior with parameters $V$ and $\phi_n=K_n/p_n$ for each $n$. In order to handle the term (\ref{concentrationComp}), we need to choose $\alpha_n$ so that $\alpha_n-K_n\rightarrow\infty$. Thus, we set $\alpha_n=K_n\log p_n$. In order to deal with the term (\ref{sparseComp}), we require $S_n-K_n\log p_n\rightarrow\infty$. Finally, to shrink the radius of $\mathcal{B}_n$, which is twice
\begin{eqnarray}
\varepsilon_n = \frac{8\sigma}{1-\delta-\theta}\sqrt{\frac{(1+\alpha_n)S_n\log p_n}{n}},
\end{eqnarray}
we may assume that $\alpha_n=K_n\log p_n$, $S_n=K_n\log^2 p_n$, and thus we need $K_n\log^2 p_n = o(\sqrt{n})$.
\end{exmp}

\subsection{Posterior contraction for absolutely continuous priors}
It is not too difficult to generalize these examples to demonstrate posterior contraction for priors with entries drawn from
\[
(1-\gamma_i)\delta_0(\beta_i)+\gamma_i\rho(\beta_i)
\]
for an arbitrary distribution $\rho$ as long as the hyperparameters $\gamma$ are sufficiently sparse and the small-ball probabilities are computable. On the other hand, the term
\[
\frac{p^{1+\tau}}{p^\alpha\sqrt{\pi\log p}\Pi(B_{C_\tau}^{\lw}(\beta^0))}
\]
is an obstruction to obtaining any contraction results for absolutely continuous priors. For example, our asymptotic framework is not applicable to the Laplace prior because the small-ball probability is proportional to the volume of a $p$-dimensional ball with radius approximately $\sqrt{\log(p)/n}$. Thus, $p^\alpha$ must grow faster than $n^{p/2}$, and hence $\alpha$ must be at least $\frac{p}{2\log n}$. The net result is that $\varepsilon$ is unbounded asymptotically. 

We conjecture that reasonable posterior contraction occurs for absolutely continuous priors with sufficient concentration on the compressible vectors. In particular, Markov's inequality provides a very poor bound for the normalization constant, and a sharper estimate should provide the correct framework for demonstrating the truth of this conjecture. It is an open problem to find general bounds for the normalization constant of any posterior -- especially bounds that can be computed efficiently from only the prior itself. Since our framework depends upon estimating the normalization constant of the posterior, the extension of our framework to absolutely continuous priors concentrating on the compressible vectors is left as an open problem.

\section*{Acknowledgements}
This work was partially funded by the Mathematics of Sensing, Exploitation, and Execution (MSEE) program (managed by Dr. Tony Falcone), by the National Science Foundation under grant DMS-1045153, and by grant R01ES17436 from the National Institute of Environmental Health Sciences (NIEHS) of the National Institutes of Health (NIH). The authors would also like to thank Mauro Maggioni for helpful discussions. We also thank the reviewers for their comments, which made this paper more complete.

\section*{Appendix A}

The proof of our main result is a modification of the argument originally devised by Schwartz \cite{Sch65}. In order to employ her strategy, we first find a large set of $y$'s for which the numerator of $\Pi(\beta\vert y)$ admits a controllable upper bound, and then we find another large set of $y$'s for which the denominator admits a controllable lower bound. 

In the literature, this former set is identified with a hypothesis test which enjoys strong consistency behavior. As is often the case, we may base this hypothesis test on a frequentist estimator, and our estimator of choice is the Dantzig selector and we employ Theorem \ref{dantzigthm} to exploit the theoretical properties of the Dantzig selector. The theoretical properties of the LASSO estimator \cite{CP07, Zhang2009} could also be exploited to form such a hypothesis test.

\begin{proof}[Proof of Theorem \ref{dantzigthm}]
Let $\beta=\sqrt{n}\beta^0$, set $h=\widetilde{\beta}-\beta$, and suppose $T_0$ and $T_{01}$ follow the precedent set in \cite{CT05}. First, we note that
\begin{equation}
\Vert h_{T_0^c}\Vert_\lw\leq\Vert h_{T_0}\Vert_\lw+2\Vert\beta_{T_0^c}\Vert_\lw.
\end{equation}
By Lemma 3.1 of \cite{CT05}, we then have
\begin{eqnarray}
\Vert h\Vert_\lt&\leq&\Vert h_{T_{01}}\Vert_\lt+S^{-1/2}\Vert h_{T_0^c}\Vert_\lw\\
&\leq &\Vert h_{T_{01}}\Vert_\lt+S^{-1/2}(\Vert h_{T_0}\Vert_\lw+2\Vert\beta_{T_0^c}\Vert_\lw)\\
&\leq &\Vert h_{T_{01}}\Vert_\lt+\Vert h_{T_0}\Vert_\lt+2S^{-1/2}\Vert\beta_{T_0^c}\Vert_\lw\\
&\leq &2\Vert h_{T_{01}}\Vert_\lt+2S^{-1/2}\Vert\beta_{T_0^c}\Vert_\lw.
\end{eqnarray}
Moreover, Lemma 3.1 also gives us
\begin{eqnarray}
\Vert h_{T_{01}}\Vert_\lt&\leq&\frac{1}{1-\delta}\Vert \widetilde{X}_{T_{01}}^T \widetilde{X} h\Vert_\lt+\frac{\theta}{1-\delta}S^{-1/2}\Vert\beta_{T_0^c}\Vert_\lw\\
&\leq&\frac{2\sqrt{2}}{1-\delta}S^{1/2}\lambda_p+\frac{\theta}{1-\delta}S^{-1/2}(\Vert h_{T_0}\Vert_\lw+2\Vert\beta_{T_0^c}\Vert_\lw)\\
&\leq&\frac{2\sqrt{2}}{1-\delta}S^{1/2}\lambda_p+\frac{\theta}{1-\delta}\Vert h_{T_0}\Vert_\lt+\frac{2\theta}{1-\delta}S^{-1/2}\Vert\beta_{T_0^c}\Vert_\lw
\end{eqnarray}
Manipulation of this last inequality yields
\begin{equation}
\Vert h_{T_{01}}\Vert_\lt\leq\frac{2\sqrt{2}}{1-\delta-\theta}\lambda_p+\frac{2\theta}{1-\delta-\theta}S^{-1/2}\Vert\beta_{T_0^c}\Vert_\lw
\end{equation}
Combining bounds, we arrive at
\begin{eqnarray}
\Vert h\Vert_\lt&\leq&\frac{4\sqrt{2}}{1-\delta-\theta}S^{1/2}\lambda_p+2\frac{1-\delta+\theta}{1-\delta-\theta}S^{-1/2}\Vert\beta_{T_0^c}\Vert_\lw\\
&=&\frac{4\sqrt{2}}{1-\delta-\theta}S^{1/2}\lambda_p+2\sqrt{n}\frac{1-\delta+\theta}{1-\delta-\theta}S^{-1/2}\Vert\beta_{T_0^c}^0\Vert_\lw\\
&\leq&\frac{4\sqrt{2}}{1-\delta-\theta}S^{1/2}\lambda_p+2\frac{1-\delta+\theta}{1-\delta-\theta}S^{-1/2}R
\end{eqnarray}
Scaling by $\sqrt{n}$ then yields the result.
\end{proof}

To simplify what follows, we set $\varepsilon$ equal to the bound in Theorem \ref{dantzigthm} and then define
\begin{eqnarray}
\mathcal{P}_{S,R}^\varepsilon=\{\beta\in\mathbb{R}^p:\Vert\beta-\beta^0\Vert_\lt>2\varepsilon\}\cap\mathcal{P}_{S,R},
\end{eqnarray}
which we shall denote as $\mathcal{P}$ when there is no possibility for ambiguity. We are now ready to define the set of $y$'s which produce controllable denominators, and we also prove the properties we shall exploit.

\begin{lem}\label{hyptest}
Define the critical region $\mathcal{C}=\{y\in\mathbb{R}^n:\Vert\hat{\beta}-\beta^0\Vert_\lt>\varepsilon\}$ and our hypothesis test is then $\Phi(y)={\bf 1}_\mathcal{C}(y)$. Then, 
\begin{enumerate}
\item $\displaystyle \mathbb{E}_{\beta^0}\Phi\leq \frac{1}{p^\alpha\sqrt{\pi\log p}}$
\item $\displaystyle \sup_{\beta\in\mathcal{P}}\mathbb{E}_\beta(1-\Phi)\leq\frac{1}{p^\alpha\sqrt{\pi\log p}}$
\end{enumerate}
\end{lem}

\begin{proof}
First, we bound the type I error rate for $\Phi$:

\begin{eqnarray}\mathbb{E}_{\beta^0}\Phi=\text{pr}_{\beta^0}(\mathcal{C})\leq\frac{1}{p^\alpha\sqrt{\pi\log p}}.\end{eqnarray}

\noindent In a similar fashion, we have

\begin{eqnarray}\sup_{\beta\in\mathcal{P}}\mathbb{E}_\beta(1-\Phi)=\sup_{\beta\in\mathcal{P}}\text{pr}_\beta\{\Vert\hat{\beta}-\beta^0\Vert_\lt\leq\varepsilon\}\end{eqnarray}

\noindent The reverse triangle inequality then yields the bound

\begin{eqnarray}\sup_{\beta\in\mathcal{P}}\mathbb{E}_\beta(1-\Phi)&\leq&\sup_{\beta\in\mathcal{P}}\text{pr}_\beta\{\Vert\hat{\beta}-\beta\Vert_\lt\geq-\varepsilon+\Vert\beta-\beta^0\Vert_\lt\}\\&\leq&\sup_{\beta\in\mathcal{P}}\text{pr}_\beta\{\Vert\hat{\beta}-\beta\Vert_\lt>\varepsilon\}\end{eqnarray}

\noindent where we have used the fact that $\Vert \beta-\beta^0\Vert_\lt>2\varepsilon$ for all $\beta\in\mathcal{P}$. Moreover, since $\beta\in\mathcal{P}$, $\text{pr}_\beta\{\Vert\hat{\beta}-\beta\Vert_\lt>\varepsilon\}\leq \frac{1}{p^\alpha\sqrt{\pi\log p}}$ by Theorem \ref{dantzigthm}, and we obtain the desired bound on the supremum. This completes the proof.
\end{proof}

Because of the behavior of the Dantzig selector, this hypothesis test is only useful for distinguishing between sparse vectors. That is, a large set of non-sparse vectors may trigger a type II error. While this may be a damning indictment for its utility as a practical hypothesis test, we merely employ $\Phi$ in the theoretical argument for our main theorem. Now, we shall prove the following more general result:

\begin{thm}\label{techthm}
Suppose $\beta^0\in\mathcal{P}=\mathcal{P}_{S,R}$ and that $\Pi$ is an arbitrary prior on $\mathbb{R}^p$.  Let $\mathcal{B}=\{\beta\in\mathbb{R}^p:\Vert\beta-\beta^0\Vert_\lt>2\varepsilon\}$, with 
\begin{eqnarray}
\varepsilon = \frac{8\sigma}{1-\delta-\theta}\sqrt{\frac{(1+\alpha)S\log p}{n}}+2\frac{1-\delta+\theta}{1-\delta-\theta}\frac{R}{\sqrt{S}},
\end{eqnarray}
and assume (A1), (A3) and (A4). For any $\alpha>0$, $\kappa>0$, $0<\nu<\alpha$, and all $u$ and $v$ satisfying $1/u+1/v=1$ with $u\geq 1$,
\begin{eqnarray}
\mathbb{E}_{\beta^0}\Pi(\mathcal{B}\vert y)&\leq& \frac{1}{p^\alpha\sqrt{\pi\log p}}\\
&&+\frac{\Pi(\mathcal{B}\setminus\mathcal{P})}{\Pi(\mathcal{D}_{\nu,\kappa})p^{-\nu}}\\
&&+\frac{1}{\Pi(\mathcal{D}_{\nu,\kappa})p^{\alpha-\nu}\sqrt{\pi\log p}}\\
&&+\text{pr}_{\beta^0}(\mathcal{A}_\kappa^c),
\end{eqnarray}
where 
\begin{eqnarray}
\mathcal{D}_{\nu,\kappa}=\left\{\beta\in\mathbb{R}^p:\Vert\beta-\beta^0\Vert_{\ell_u}<C_{\nu,\kappa}\right\},
\end{eqnarray}
\begin{equation}
C_{\nu,\kappa}=\left(\frac{\sqrt{2\Vert X^T X\Vert_{\ell_u\rightarrow\ell_v}\sigma^2\nu\log p}}{\kappa+\sqrt{\kappa^2+2\Vert X^T X\Vert_{\ell_u\rightarrow\ell_v}\sigma^2\nu\log p}}\right)\sqrt{\frac{2\sigma^2\nu\log p}{\Vert X^T X\Vert_{\ell_u\rightarrow\ell_v}}},
\end{equation}
and
\begin{equation}
\mathcal{A}_\kappa=\{y\in\mathbb{R}^n:\Vert X^T(y-X\beta^0)\Vert_{\ell_v}\leq\kappa\}.
\end{equation}
\end{thm}

We recover Theorem \ref{mainthm} when we set $u=1$, $v=\infty$, $\kappa=\sigma\sqrt{n}\sqrt{2(1+\tau)\log p}$, and $\nu=1+\tau$.

\begin{proof}[Proof of Theorem \ref{mainthm}]
We apply the standard divide-and-conquer strategy originally devised by Schwartz \cite{Sch65} to obtain
\begin{eqnarray}
\Pi(\mathcal{B}\vert y)&=&\Phi(y)\Pi(\mathcal{B}\vert y)+(1-\Phi(y))\Pi(\mathcal{B}\vert y){\bf 1}_{\mathcal{A}_\kappa}(y)\\
&&+(1-\Phi(y))\Pi(\mathcal{B}\vert y){\bf 1}_{\mathcal{A}_\kappa^c}(y)\\
&\leq&\Phi(y)+(1-\Phi(y))\Pi(\mathcal{B}\vert y){\bf 1}_{\mathcal{A}_\kappa}(y)+{\bf 1}_{\mathcal{A}_\kappa^c}(y)
\end{eqnarray}
By Lemma 1, we have that $\mathbb{E}_{\beta_0}\Phi<\frac{1}{p^\alpha\sqrt{\pi\log p}}$, so this term is immediately eliminated. Additionally, we have that $\mathbb{E}_{\beta^0}{\bf 1}_{\mathcal{A}_\kappa^c}(y)=\text{pr}_{\beta^0}(\mathcal{A}_\kappa^c)$. Having dispensed with the first and third terms, we proceed to attack the middle term.

We first multiply this remaining term by a form of $1$ to obtain

\begin{eqnarray}\frac{(1-\Phi(y))\int_\mathcal{B}\frac{f(y\vert\beta)}{f(y\vert\beta^0)}d\Pi(\beta)}{\int \frac{f(y\vert\beta)}{f(y\vert\beta^0)}d\Pi(\beta)}{\bf 1}_{\mathcal{A}_\kappa}\end{eqnarray}

\noindent Now, we bound the denominator by the expression

\begin{eqnarray}\int\frac{f(y\vert\beta)}{f(y\vert\beta^0)}d\Pi(\beta)\geq \exp\{-\nu\log p\}\Pi(\mathcal{D}_\nu(y))\end{eqnarray}

\noindent where

\begin{eqnarray}
\mathcal{D}_\nu(y)&=&\{\beta\in\mathbb{R}^p:\frac{1}{\log p}\log\frac{f(y\vert\beta^0)}{f(y\vert\beta)}<\nu\}\\
&=&\{\beta\in\mathbb{R}^p:\frac{1}{\log p}(\Vert y-X\beta\Vert_\lt^2-\Vert y-X\beta^0\Vert_\lt^2)<2\sigma^2\nu\}\\
&=&\{\beta\in\mathbb{R}^p:\Vert y-X\beta\Vert_\lt^2-\Vert y-X\beta^0\Vert_\lt^2<2\sigma^2\nu\log p\}\label{eq01}
\end{eqnarray}

\noindent By applying the H\"{o}lder inequality and the definition of the operator norm, it is easy to see that the left-hand side of the inequality in (\ref{eq01}) is
\begin{eqnarray}
&=&\langle (y-X\beta)+(y-X\beta^0),(y-X\beta)-(y-X\beta^0)\rangle\\
&=&\langle 2y-2X\beta^0,X(\beta^0-\beta)\rangle +\langle X(\beta^0-\beta),X(\beta^0-\beta)\rangle\\
&\leq&2\Vert X^T(y-X\beta^0)\Vert_{\ell_v}\Vert \beta-\beta^0\Vert_{\ell_u}+\Vert X^T X\Vert_{\ell_u\rightarrow\ell_v}\Vert\beta-\beta^0\Vert_{\ell_u}^2\\
&\leq&2\kappa\Vert \beta-\beta^0\Vert_{\ell_u}+\Vert X^T X\Vert_{\ell_u\rightarrow\ell_v}\Vert\beta-\beta^0\Vert_{\ell_u}^2\end{eqnarray}
\noindent since $\kappa\geq\Vert X^T(y-X\beta^0)\Vert_{\ell_v}$ for $y\in\mathcal{A}_\kappa$. Following (\ref{eq01}), we force $\Vert \beta-\beta^0\Vert_\lt$ to satisfy the inequality 

\begin{eqnarray}
\Vert X^T X\Vert_{\ell_u\rightarrow\ell_v}\Vert\beta-\beta^0\Vert_{\ell_u}^2+2\kappa\Vert\beta-\beta^0\Vert_{\ell_u}<2\sigma^2\nu\log p.
\end{eqnarray}

\noindent It is not difficult to establish that the following bound on $\Vert\beta-\beta^0\Vert_{\ell_u}$ implies this previous inequality
\small
\begin{eqnarray}
&<&\frac{\sqrt{4\kappa^2+8\Vert X^T X\Vert_{\ell_u\rightarrow\ell_v}\sigma^2\nu\log p}-2\kappa}{2\Vert X^T X\Vert_{\ell_u\rightarrow\ell_v}}\\
&=&\left(\frac{\sqrt{8\Vert X^T X\Vert_{\ell_u\rightarrow\ell_v}\sigma^2\nu\log p}}{2\kappa+\sqrt{4\kappa^2+8\Vert X^T X\Vert_{\ell_u\rightarrow\ell_v}\sigma^2\nu\log p}}\right)\frac{\sqrt{8\Vert X^T X\Vert_{\ell_u\rightarrow\ell_v}\sigma^2\nu\log p}}{2\Vert X^T X\Vert_{\ell_u\rightarrow\ell_v}}\\
&=&\left(\frac{\sqrt{2\Vert X^T X\Vert_{\ell_u\rightarrow\ell_v}\sigma^2\nu\log p}}{\kappa+\sqrt{\kappa^2+2\Vert X^T X\Vert_{\ell_u\rightarrow\ell_v}\sigma^2\nu\log p}}\right)\sqrt{\frac{2\sigma^2\nu\log p}{\Vert X^T X\Vert_{\ell_u\rightarrow\ell_v}}}.
\end{eqnarray}
\normalsize
\noindent Based on this sequence of inequalities, we conclude that $\mathcal{D}_{\nu,\kappa}\subset\mathcal{D}_\nu(y)$ when $y\in\mathcal{A}_\kappa$. Putting this all together, we have that $\Pi(\mathcal{D}_\nu(y))\geq\Pi(\mathcal{D}_{\nu,\kappa})$ for $y\in\mathcal{A}_\kappa$, and hence
\begin{eqnarray}\int\frac{f(y\vert\beta)}{f(y\vert\beta^0)}d\Pi(\beta)\geq p^{-\nu}\Pi(\mathcal{D}_{\nu,\kappa})\end{eqnarray}

\noindent for all $y\in\mathcal{A}_\kappa$. Applying this, we obtain the bound

\begin{eqnarray}
(1-\Phi(y))\Pi(\mathcal{B}){\bf 1}_{\mathcal{A}_\kappa}(y)\leq \frac{(1-\Phi(y))\int_\mathcal{B} \frac{f(y\vert\beta)}{f(y\vert\beta^0)}d\Pi(\beta)}{p^{-\nu}\Pi(\mathcal{D}_{\nu,\kappa})}.\end{eqnarray}

\noindent Taking the expectation of the numerator and applying Tonelli yields

\begin{eqnarray}
\mathbb{E}_{\beta^0}(1-\Phi(y))\int_\mathcal{B} \frac{f(y\vert\beta)}{f(y\vert\beta^0)}d\Pi(\beta)&=&\int_\mathcal{B} \mathbb{E}_{\beta_0}(1-\Phi(y))\frac{f(y\vert\beta)}{f(y\vert\beta^0)}d\Pi(\beta)\\
&=&\int_\mathcal{B} \mathbb{E}_{\beta}(1-\Phi(y))d\Pi(\beta)
\end{eqnarray}
We now split this and bound using Lemma 1:
\begin{eqnarray}
\int_\mathcal{B} \mathbb{E}_{\beta}(1-\Phi(y))d\Pi(\beta)&=&\int_{\mathcal{B}\setminus\mathcal{P}}\mathbb{E}_{\beta}(1-\Phi(y))d\Pi(\beta)\\
&&+\int_{\mathcal{P}}\mathbb{E}_{\beta}(1-\Phi(y))d\Pi(\beta)\\
&\leq&\Pi(\mathcal{B}\setminus\mathcal{P})+\Pi(\mathcal{P})\frac{1}{p^\alpha\sqrt{\pi\log p}}\\
&\leq&\Pi(\mathcal{B}\setminus\mathcal{P})+\frac{1}{p^\alpha\sqrt{\pi\log p}}
\end{eqnarray}
This establishes the result.
\end{proof}

\section*{Appendix B}

In order to prove Theorem \ref{UG}, we shall require some additional notation. For a fixed $\sigma, V, X, y, e,$ and $\gamma\in\{0,1\}^p$, we let $X_\gamma$ denote the matrix obtained by deleting the columns of $X$ with indices $i$ such that $\gamma_i=0$, $P_\gamma$ denote the orthogonal projection onto the span of the columns of $X_\gamma$,
\begin{equation}
\Sigma_\gamma=\left(\frac{1}{\sigma^2}X_\gamma^T X_\gamma+\frac{1}{V^2}I_{S\times S}\right)^{-1/2},
\end{equation}
and
\begin{equation}
\mu_\gamma=\frac{1}{\sigma^2}\Sigma_\gamma^{2}X_\gamma^Ty. 
\end{equation}
Additionally, we shall slightly abuse notation by letting $\beta_\gamma$ denote the projection of $\beta$ onto the coordinates indicated by $\gamma$ and the Hadamard product of $\beta$ with $\gamma$ depending upon the context. Finally, for $\gamma,\gamma^\prime\in\{0,1\}^p$, we shall write $\gamma\leq\gamma^\prime$ to indicate that $\gamma^\prime$ dominates $\gamma$ entry wise.

We first begin with a simple probabilistic noise bound in the spirit of Cand\`{e}s and Tao \cite{CT05}. The proof is a simple application of the Markov inequality.

\begin{lem}\label{lem:noise}
Assuming that $e_i\sim\mathcal{N}(0,\sigma^2)$ for $i=1,\ldots, n$ and $\Vert \widetilde{X}_i\Vert_{\lt}^2=1$ for $i=1,\ldots, p$. If 
\begin{equation}
\mathcal{E}=\{e\in\mathbb{R}^n:\Vert \widetilde{X}_\gamma^Te\Vert_{\lt}^2\leq4\sigma^2(1+\alpha)\vert\gamma\vert\log p, \forall \gamma\in\{0,1\}^p\},\label{noiseEvent}
\end{equation}
then
\begin{eqnarray}
\rm{pr}_e(\mathcal{E})>1-\frac{1}{p^{\alpha}\sqrt{\pi\log p}}.
\end{eqnarray}
\end{lem}

The next lemma we shall require deterministically bounds the difference between similar operators.

\begin{lem}\label{lem:projappx}
Assume (A1)-(A4), if  $\gamma\in\{0,1\}^p$ satisfies $\gamma^0\leq\gamma$ and $\vert\gamma\vert\leq 2S$, we have that \begin{eqnarray}
\left\Vert P_{\gamma}-X_{\gamma}(X_\gamma^TX_\gamma+\frac{\sigma^2}{V^2}I_{\vert\gamma\vert\times\vert\gamma\vert})^{-1}X_{\gamma}^T\right\Vert_{\lt\rightarrow\lt}\leq\frac{\sigma^2}{n(1-\delta)V^2+\sigma^2}.
\end{eqnarray}
and
\begin{eqnarray}
\left\Vert I_{\vert\gamma\vert\times\vert\gamma\vert}-(X_\gamma^TX_\gamma+\frac{\sigma^2}{V^2}I_{\vert\gamma\vert\times\vert\gamma\vert})^{-1}X_{\gamma}^TX_{\gamma}\right\Vert_{\lt\rightarrow\lt}\leq\frac{\sigma^2}{n(1-\delta)V^2+\sigma^2}.
\end{eqnarray}
\end{lem}

We shall also require bounds on the determinants of the restricted operators. This lemma and the preceding lemma follow from the RIP hypothesis and application of an SVD.

\begin{lem}\label{lem:det}
Assuming (A1) and (A4), if $\gamma\in\{0,1\}^p$ satisfies $\vert\gamma\vert\leq2S$, then we have that
\begin{eqnarray}
\left(\frac{n(1+\delta)}{\sigma^2}+\frac{1}{V^2}\right)^{-\vert\gamma\vert/2}\leq\det(\Sigma_\gamma)\leq\left(\frac{n(1-\delta)}{\sigma^2}+\frac{1}{V^2}\right)^{-\vert\gamma\vert/2}.
\end{eqnarray}
On the other hand, if $\vert\gamma\vert>2S$, then
\begin{eqnarray}
\det(\Sigma_\gamma)\leq\left(\frac{n(1-\delta)}{\sigma^2}+\frac{1}{V^2}\right)^{-S}.
\end{eqnarray}
\end{lem}

Now, we exhibit a bound on the difference between norms of different reconstructions.

\begin{lem}\label{lem:term1}
Assume (A1), (A2), and (A4). If $\gamma,\gamma^\prime\in\{0,1\}^p$ satisfy $\gamma^0\leq\gamma$ and $\vert\gamma^\prime\vert\leq 2S$, then
\begin{eqnarray}
(X\beta^0)^T(P_{\gamma^\prime}-P_{\gamma})X\beta^0\leq-n\left(1-\delta-\frac{\theta^2}{1-\delta}\right)\Vert\beta^0_{\gamma^0\setminus\gamma^\prime}\Vert_{\lt}^2,
\end{eqnarray}
where $\gamma^0\setminus\gamma^\prime=\{i\in[p]:\gamma^0_i=1-\gamma^\prime_i=1\}$. Moreover, $\delta+\frac{\theta^2}{1-\delta}<1$.
\end{lem}
\begin{proof}
Since $\gamma^0\leq\gamma$, we have that $P_\gamma X\beta^0=X_{\gamma_0}\beta_{\gamma^0}^0$, and therefore
\begin{eqnarray}
\left(X\beta^0\right)^T(P_{\gamma^\prime}-P_\gamma)X\beta^0&=&\left(X_{\gamma^0}\beta_{\gamma^0}^0\right)^T(P_{\gamma^\prime}-I_{n\times n})X_{\gamma_0}\beta_{\gamma^0}\\
&=&\left(X_{\gamma^0\setminus\gamma^\prime}\beta_{\gamma^0\setminus\gamma^\prime}^0\right)^T(P_{\gamma^\prime}-I_{n\times n})X_{\gamma_0\setminus\gamma^\prime}\beta_{\gamma^0\setminus\gamma^\prime}\nonumber\\
&=&-\Vert X_{\gamma^0\setminus\gamma^\prime}\beta_{\gamma^0\setminus\gamma^\prime}^0\Vert_{\lt}^2+\Vert P_{\gamma^\prime} X_{\gamma^0\setminus\gamma^\prime}\beta_{\gamma^0\setminus\gamma^\prime}^0\Vert_{\lt}^2.\nonumber
\end{eqnarray}
We then have
\begin{eqnarray}
\Vert P_{\gamma^\prime} X_{\gamma^0\setminus\gamma^\prime}\beta_{\gamma^0\setminus\gamma^\prime}^0\Vert_{\lt}^2&\leq&\frac{n}{1-\delta}\Vert \widetilde{X}_{\gamma^\prime}^T \widetilde{X}_{\gamma^0\setminus\gamma^\prime}\beta_{\gamma^0\setminus\gamma^\prime}^0\Vert_{\lt}^2\\
&\leq&n\frac{\theta^2}{1-\delta}\Vert\beta_{\gamma^0\setminus\gamma^\prime}^0\Vert_{\lt}^2
\end{eqnarray}
since
\begin{eqnarray}
\Vert \widetilde{X}_\gamma^T \widetilde{X}_{\gamma^0\setminus\gamma}\beta_{\gamma^0\setminus\gamma}^0\Vert_{\lt}^2&=&\langle \widetilde{X}_\gamma \widetilde{X}_\gamma^T\widetilde{X}_{\gamma^0\setminus\gamma}\beta_{\gamma^0\setminus\gamma}^0,\widetilde{X}_{\gamma^0\setminus\gamma}\beta_{\gamma^0\setminus\gamma}^0\rangle\\
&\leq&\theta\Vert \widetilde{X}_\gamma^T \widetilde{X}_{\gamma^0\setminus\gamma}\beta_{\gamma^0\setminus\gamma}^0\Vert_{\lt}\Vert\beta_{\gamma^0\setminus\gamma}^0\Vert_{\lt}
\end{eqnarray}
implies $\Vert \widetilde{X}_\gamma^T \widetilde{X}_{\gamma^0\setminus\gamma}\beta_{\gamma^0\setminus\gamma}^0\Vert_{\lt}^2\leq\theta^2\Vert\beta_{\gamma^0\setminus\gamma}^0\Vert_{\lt}^2$.  Combining this with the fact that $n(1-\delta)\Vert\beta_{\gamma^0\setminus\gamma^\prime}^0\Vert_{\lt}^2 \leq\Vert X_{\gamma^0\setminus\gamma^\prime}\beta_{\gamma^0\setminus\gamma^\prime}^0\Vert_{\lt}^2$, we obtain the bound
\begin{eqnarray}
\left(X\beta^0\right)^T(P_{\gamma^\prime}-P_\gamma)X\beta^0 \leq -n\left(1-\delta-\frac{\theta^2}{1-\delta}\right)\Vert\beta_{\gamma^0\setminus\gamma}^0\Vert_{\lt}^2.
\end{eqnarray}
Finally, note that
\begin{eqnarray}
\theta+\delta<1\Longrightarrow\theta^2<(1-\delta)^2\Longrightarrow\frac{\theta^2}{1-\delta}\leq 1-\delta\Longrightarrow\delta+\frac{\theta^2}{1-\delta}<1.
\end{eqnarray}
\end{proof}

This next lemma bounds the differences of inner products of reconstructions with the noise vector. 

\begin{lem}\label{lem:term2}
Assume (A1) through (A4), and $e\in\mathcal{E}$ from (\ref{noiseEvent}). If $\gamma^0\leq\gamma$ and $\vert\gamma^\prime\vert\leq 2S$, then
\begin{equation}
\vert (X\beta^0)^T(P_{\gamma^\prime}-P_{\gamma})e\vert \leq 2\sigma\frac{1-\delta+\theta}{1-\delta}\sqrt{(1+\alpha)\max\{\vert\gamma^0\setminus\gamma^\prime\vert,\vert\gamma^\prime\vert\} n\log p}\Vert\beta^0_{\gamma^0\setminus\gamma^\prime}\Vert_{\lt}.\nonumber
\end{equation}
\end{lem}
\begin{proof}
We compute
\begin{eqnarray}
\vert(X\beta^0)^T(P_{\gamma^\prime}-P_{\gamma})e\vert&=&\vert(X_{\gamma^0}\beta_{\gamma^0}^0)^T(P_{\gamma^\prime}-I_{n\times n})e\vert\\
&=&\vert(X_{\gamma^0\setminus\gamma}\beta_{\gamma^0\setminus\gamma}^0)^T(P_{\gamma^\prime}-I_{n\times n})e\vert\\
&=&\vert(\beta_{\gamma^0\setminus\gamma^\prime}^0)^TX_{\gamma^0\setminus\gamma^\prime}^Te+(X_{\gamma^0\setminus\gamma^\prime}\beta_{\gamma^0\setminus\gamma^\prime}^0)^TP_{\gamma^\prime}e\vert\\
&\leq&\Vert X_{\gamma^0\setminus\gamma}^Te\Vert_{\lt}\Vert\beta_{\gamma^0\setminus\gamma}^0\Vert_{\lt}+\Vert P_{\gamma^\prime}X_{\gamma^0\setminus\gamma^\prime}^Te\Vert_{\lt}\Vert\beta_{\gamma^0\setminus\gamma^\prime}^0\Vert_{\lt}\nonumber\\
&\leq&\Vert X_{\gamma^0\setminus\gamma^\prime}^Te\Vert_{\lt}\Vert\beta_{\gamma^0\setminus\gamma^\prime}^0\Vert_{\lt}+\frac{\theta}{1-\delta}\Vert X_{\gamma^\prime}^Te\Vert_{\lt}\Vert\beta_{\gamma^0\setminus\gamma^\prime}^0\Vert_{\lt}\nonumber\\
&\leq&2\sigma\left(1+\frac{\theta}{1-\delta}\right)\sqrt{(1+\alpha)n\max\{\vert\gamma^0\setminus\gamma^\prime\vert,\vert\gamma^\prime\vert\}\log p}\Vert\beta_{\gamma^0\setminus\gamma^\prime}^0\Vert_{\lt}.\nonumber
\end{eqnarray}
\end{proof}

Our last lemma is used to clean up a calculation that arises.

\begin{lem}\label{lem:cleanup}
Suppose
\begin{equation}
\tau>A\sqrt{\frac{(1+\alpha)S\log p}{n}}=\frac{8\sqrt{2}\sigma}{1-\delta-\theta}\sqrt{\frac{(1+\alpha)S\log p}{n}}.
\end{equation}
Then
\begin{equation}
-n\left(1-\delta-\frac{\theta^2}{1-\delta}\right)\tau^2+4\sqrt{2}\sigma\frac{1-\delta+\theta}{1-\delta}\sqrt{(1+\alpha)Sn\log p}\tau\leq-\frac{n}{2}\left(1-\delta-\frac{\theta^2}{1-\delta}\right)\tau^2.\nonumber
\end{equation}
\end{lem}

\begin{proof}[Proof of Theorem \ref{UG}]
First, we exhibit an explicit formula for the posterior. Given any measurable set $U\subset\mathbb{R}^n$, we have that
\begin{eqnarray}
&&\int_U f(y\vert\beta)d\Pi(\beta)\\
&=& \binom{p}{S}^{-1}\sum_{\gamma\in\{0,1\}_S^p}\int_U (2\pi\sigma^2)^{-n/2}e^{-\Vert y-X\beta_\gamma\Vert_\lt^2/2\sigma^2}(2\pi V^2)^{-S/2} e^{-\Vert \beta_\gamma\Vert_\lt^2/2V^2}d\beta_\gamma\nonumber\\
&=& \binom{p}{S}^{-1}(2\pi\sigma^2)^{-n/2}(2\pi V^2)^{-S/2}\sum_{\gamma\in\{0,1\}_S^p}\int_U e^{-\frac{\Vert y-X\beta_\gamma\Vert_\lt^2}{2\sigma^2}-\frac{\Vert\beta_\gamma\Vert_\lt^2}{2V^2}}d\beta_\gamma
\end{eqnarray}
\noindent Completing the square gives us
\begin{equation}
\frac{\Vert y-X\beta_\gamma\Vert_\lt^2}{2\sigma^2}+\frac{\Vert\beta_\gamma\Vert_\lt^2}{2V^2}=\frac{1}{2}(\beta_\gamma-\mu_\gamma)^T\Sigma_\gamma^{-2}(\beta_\gamma-\mu_\gamma)+\frac{1}{2\sigma^2}\Vert y\Vert_\lt^2-\frac{1}{2}\mu_\gamma^T\Sigma_\gamma^{-2}\mu_\gamma\nonumber
\end{equation}
Therefore, if $U=\mathbb{R}^n$, we have that
\begin{equation}
\int_U e^{-\frac{\Vert y-X\beta_\gamma\Vert_\lt^2}{2\sigma^2}-\frac{\Vert\beta_\gamma\Vert_\lt^2}{2V^2}}d\beta_\gamma= (2\pi)^{S/2}\text{det}(\Sigma_\gamma) e^{-\frac{1}{2\sigma^2}\Vert y\Vert_\lt^2+\frac{1}{2}\mu_\gamma^T\Sigma_\gamma^{-2}\mu_\gamma}
\end{equation}
and hence
\begin{eqnarray}
\int_{\mathbb{R}^n} f(y\vert\beta)d\Pi(\beta)&=&e^{-\frac{1}{2\sigma^2}\Vert y\Vert_\lt^2}\binom{p}{S}^{-1}(2\pi\sigma^2)^{-n/2}V^{-S}\sum_{\gamma\in\{0,1\}_S^p}\text{det}(\Sigma_\gamma) e^{\frac{1}{2}\mu_\gamma^T\Sigma_\gamma^{-2}\mu_\gamma}\nonumber
\end{eqnarray}
On the other hand,\small
\begin{eqnarray}
&&\int_{B_{2\varepsilon}^{\lt}(\beta^0)}f(y\vert\beta)d\Pi(\beta)\\
&=&\binom{p}{S}^{-1}(2\pi\sigma^2)^{-n/2}(2\pi V^2)^{-S/2}\sum_{\gamma\in\{0,1\}_S^p}\int_{B_{2\varepsilon}^{\lt}(\beta^0)} e^{-\frac{\Vert y-X\beta_\gamma\Vert_\lt^2}{2\sigma^2}-\frac{\Vert\beta_\gamma\Vert_\lt^2}{2V^2}}d\beta_\gamma\nonumber\\
&=&e^{-\frac{1}{2\sigma^2}\Vert y\Vert_\lt^2}\binom{p}{S}^{-1}(2\pi\sigma^2)^{-n/2}V^{-S}\sum_{\gamma\in\{0,1\}_S^p}\det(\Sigma_\gamma)e^{\frac{1}{2}\mu_\gamma^T\Sigma_\gamma^{-2}\mu_\gamma}\int_{B_{2\varepsilon}^{\lt}(\beta^0)} \frac{e^{-\frac{1}{2}(\beta_\gamma-\mu_\gamma)^T\Sigma_\gamma^{-2}(\beta_\gamma-\mu_\gamma)}}{\sqrt{2\pi}^S\det(\Sigma_\gamma)}d\beta_\gamma\nonumber
\end{eqnarray}
\normalsize
Putting this all together, we have
\begin{eqnarray}
\Pi(B_{2\varepsilon}^{\lt}(\beta^0)\vert y)&=&\frac{\displaystyle\int_{B_{2\varepsilon}^{\lt}(\beta^0)}f(y\vert\beta)d\Pi(\beta)}{\displaystyle\int_{\mathbb{R}^n} f(y\vert\beta)d\Pi(\beta)}\\
&=&\frac{\displaystyle\sum_{\gamma\in\{0,1\}_S^p}\det(\Sigma_\gamma)e^{\frac{1}{2}\mu_\gamma^T\Sigma_\gamma^{-2}\mu_\gamma}\int_{B_{2\varepsilon}^{\lt}(\beta^0)} \frac{e^{-\frac{1}{2}(\beta_\gamma-\mu_\gamma)^T\Sigma_\gamma^{-2}(\beta_\gamma-\mu_\gamma)}}{\sqrt{2\pi}^S\det(\Sigma_\gamma)}d\beta_\gamma}
{\displaystyle\sum_{\gamma\in\{0,1\}_S^p}\text{det}(\Sigma_\gamma) e^{\frac{1}{2}\mu_\gamma^T\Sigma_\gamma^{-2}\mu_\gamma}}\nonumber.
\end{eqnarray}
Now that we have an explicit expression for the posterior, we bound this expression below by reducing the sum in the numerator to the indices in the set
\begin{equation}
G=\left\{\gamma\in\{0,1\}_S^p:\Vert\beta^0_{\gamma^0\setminus\gamma}\Vert_{\lt}\leq A\sqrt{\frac{(1+\alpha)S\log p}{n}}\right\}
\end{equation}
That is, we restrict to the indices that capture most of the mass of $\beta^0$. If $\gamma\in G$, then
\begin{eqnarray}
\Vert\beta_\gamma^0-\mu_{\gamma}\Vert_{\lt}&=&\Vert\beta_\gamma^0-(X_\gamma^TX_\gamma+\frac{\sigma^2}{V^2}I_{S\times S})^{-1}X_\gamma^T(X\beta^0+e)\Vert_{\lt}\\
&\leq&\Vert(I_{S\times S}-(X_\gamma^TX_\gamma+\frac{\sigma^2}{V^2}I_{S\times S})^{-1}X_\gamma^TX_\gamma)\beta_\gamma^0\Vert_{\lt}\\
&&+\Vert(X_\gamma^TX_\gamma+\frac{\sigma^2}{V^2}I_{S\times S})^{-1}X_\gamma^TX_{\gamma^0\setminus\gamma}\beta_{\gamma^0\setminus\gamma}^0\Vert_{\lt}\\
&&+\Vert(X_\gamma^TX_\gamma+\frac{\sigma^2}{V^2}I_{S\times S})^{-1}X_\gamma^Te\Vert_{\lt}\\
&\leq&\frac{\sigma^2}{n(1-\delta)V^2+\sigma^2}\Vert\beta_\gamma^0\Vert_{\lt}\\
&&+\frac{n\theta}{n(1-\delta)+\sigma^2/V^2}\Vert\beta_{\gamma^0\setminus\gamma}^0\Vert_{\lt}\\
&&+2\frac{\sqrt{nS(1+\alpha)\sigma^2\log p}}{n(1-\delta)+\sigma^2/V^2}\\
&=&\frac{\sigma^2\Vert\beta_\gamma^0\Vert_{\lt}+n\theta V^2\Vert\beta_{\gamma^0\setminus\gamma}^0\Vert_{\lt}+2\sigma V^2\sqrt{(1+\alpha)Sn\log p}}{n(1-\delta)V^2+\sigma^2}.\nonumber
\end{eqnarray}
Given this bound, we may conclude that (using the bound $\Vert\beta_\gamma^0\Vert_{\lt}\leq C\sqrt{S}$, but note that a large enough $V$ may be chosen to dampen this contribution)
\begin{eqnarray}
\Vert\beta^0-\mu_\gamma\Vert_{\lt}&\leq&\Vert\beta_{\gamma^0\setminus\gamma}^0\Vert_{\lt}+\Vert\beta_\gamma^0-\mu_\gamma\Vert_{\lt}\\
&\leq& \frac{\sigma^2\Vert\beta_\gamma^0\Vert_{\lt}+(n(1-\delta+\theta)V^2+\sigma^2)\Vert\beta_{\gamma^0\setminus\gamma}^0\Vert_{\lt}+2\sigma V^2\sqrt{(1+\alpha)Sn\log p}}{n(1-\delta)V^2+\sigma^2}\nonumber\\
&\leq&\varepsilon
\end{eqnarray}
and hence
\begin{eqnarray}
\int_{B_{2\varepsilon}^{\lt}(\beta^0)} \frac{e^{-\frac{1}{2}(\beta_\gamma-\mu_\gamma)^T\Sigma_\gamma^{-2}(\beta_\gamma-\mu_\gamma)}}{\sqrt{2\pi}^S\det(\Sigma_\gamma)}d\beta_\gamma&\geq&\int_{B_{\varepsilon}^{\lt}(\mu_\gamma)} \frac{e^{-\frac{1}{2}(\beta_\gamma-\mu_\gamma)^T\Sigma_\gamma^{-2}(\beta_\gamma-\mu_\gamma)}}{\sqrt{2\pi}^S\det(\Sigma_\gamma)}d\beta_\gamma\nonumber\\
&\geq&\int_{B_{\varepsilon}^{\lt}(\mu_{\gamma})}\frac{e^{-\frac{1}{2}(n\frac{1-\delta}{\sigma^2}+\frac{1}{V^2})\Vert\beta_{\gamma}-\mu_{\gamma}\Vert_{\lt}^2}}{\sqrt{2\pi\left(n\frac{1-\delta}{\sigma^2}+\frac{1}{V^2}\right)^{-1}}^S}d\beta_{\gamma}\nonumber\\
&\geq&1-e^{-\frac{1}{4}(n\frac{1-\delta}{\sigma^2}+\frac{1}{V^2})\varepsilon^2}.
\end{eqnarray}
This last expression holds because of the hypothesis $(n\frac{1-\delta}{\sigma^2}+\frac{1}{V^2})\varepsilon^2\geq S/2$. 
At this stage, we have constructed the bound
\begin{equation}
\Pi(B_{2\varepsilon}^{\lt}(\beta^0)\vert y)\geq \frac{\sum_{\gamma\in G}\det(\Sigma_\gamma)e^{\frac{1}{2}\mu_\gamma^T\Sigma_\gamma^{-2}\mu_\gamma}}{\sum_{\gamma\in\{0,1\}_S^p}\det(\Sigma_\gamma)e^{\frac{1}{2}\mu_\gamma^T\Sigma_\gamma^{-2}\mu_\gamma}}(1-e^{-\frac{1}{4}(n\frac{1-\delta}{\sigma^2}+\frac{1}{V^2})\varepsilon^2}).
\end{equation}
We now approach the expression
\begin{eqnarray}
&&\frac{\sum_{\gamma\in G}\det(\Sigma_\gamma)e^{\frac{1}{2}\mu_\gamma^T\Sigma_\gamma^{-2}\mu_\gamma}}{\sum_{\gamma\in\{0,1\}_S^p}\det(\Sigma_\gamma)e^{\frac{1}{2}\mu_\gamma^T\Sigma_\gamma^{-2}\mu_\gamma}}\label{eq:UGdenominator}\\
&=&\frac{1}{\displaystyle1+\frac{\sum_{\gamma\in\{0,1\}_S^p\setminus G}\det(\Sigma_\gamma)e^{\frac{1}{2}\mu_\gamma^T\Sigma_\gamma^{-2}\mu_\gamma}}{\sum_{\gamma\in G}\det(\Sigma_\gamma)e^{\frac{1}{2}\mu_\gamma^T\Sigma_\gamma^{-2}\mu_\gamma}}}\nonumber\\
&=&\frac{1}{\displaystyle1+\sum_{\gamma^\prime\in\{0,1\}_S^p\setminus G}\frac{1}{\sum_{\gamma\in G}\frac{\det(\Sigma_\gamma)}{\det(\Sigma_{\gamma^\prime})}e^{\frac{1}{2}(\mu_\gamma^T\Sigma_\gamma^{-2}\mu_\gamma-\mu_{\gamma^\prime}^T\Sigma_{\gamma^\prime}^{-2}\mu_{\gamma^\prime})}}}\nonumber\\
&\geq&\frac{1}{\displaystyle1+\sum_{\gamma^\prime\in\{0,1\}_S^p\setminus G}\frac{1}{\sum_{\gamma^0\leq\gamma}\frac{\det(\Sigma_\gamma)}{\det(\Sigma_{\gamma^\prime})}e^{\frac{1}{2}(\mu_\gamma^T\Sigma_\gamma^{-2}\mu_\gamma-\mu_{\gamma^\prime}^T\Sigma_{\gamma^\prime}^{-2}\mu_{\gamma^\prime})}}}.\nonumber
\end{eqnarray}
In the last step, we have reduced the index set over the sum inside the continued fraction from $G$ to its subset $\{\gamma\in\{0,1\}_S^p:\gamma^0\leq\gamma\}$. Based on this initial bound, we shall seek upper bounds on the expressions
\begin{equation}
\frac{\det(\Sigma_{\gamma^\prime})}{\det(\Sigma_{\gamma})}e^{\frac{1}{2}(\mu_{\gamma^\prime}^T\Sigma_{\gamma^\prime}^{-2}\mu_{\gamma^\prime}-\mu_\gamma^T\Sigma_\gamma^{-2}\mu_\gamma)}
\end{equation}
for all $\gamma^0\leq\gamma$ and $\gamma^\prime\in\{0,1\}_S^p\setminus G$. First, we note that 
\begin{eqnarray}
\frac{\det(\Sigma_{\gamma^\prime})}{\det(\Sigma_{\gamma})}\leq\left(\frac{n(1+\delta)V^2+\sigma^2}{n(1-\delta)V^2+\sigma^2}\right)^{S/2}
\end{eqnarray}
by Lemma \ref{lem:det}. The remaining expressions that we must examine have the form
\begin{eqnarray}
\exp\left\{\frac{1}{2}\left(\mu_{\gamma^\prime}^T\Sigma_{\gamma^\prime}^{-2}\mu_{\gamma^\prime}-\mu_{\gamma}^T\Sigma_{\gamma}^{-2}\mu_{\gamma}\right)\right\}.
\end{eqnarray}
Since
\begin{eqnarray}
\mu_{\gamma^\prime}^T\Sigma_{\gamma^\prime}^{-2}\mu_{\gamma^\prime}-\mu_{\gamma}^T\Sigma_{\gamma}^{-2}\mu_{\gamma}&=&\left(\frac{1}{\sigma^2}\Sigma_{\gamma^\prime}^2X_{\gamma^\prime}^Ty\right)^T\Sigma_{\gamma^\prime}^{-2}\left(\frac{1}{\sigma^2}\Sigma_{\gamma^\prime}^2X_{\gamma^\prime}^Ty\right)\\
&&-\left(\frac{1}{\sigma^2}\Sigma_{\gamma}^2X_{\gamma}^Ty\right)^T\Sigma_{\gamma}^{-2}\left(\frac{1}{\sigma^2}\Sigma_{\gamma}^2X_{\gamma}^Ty\right)\\
&=&\frac{1}{\sigma^2}y^TX_{\gamma^\prime}(X_{\gamma^\prime}^TX_{\gamma^\prime}+\frac{\sigma^2}{V^2}I_{\vert\gamma^\prime\vert\times\vert\gamma^\prime\vert})^{-1}X_{\gamma^\prime}^Ty\nonumber\\
&&-\frac{1}{\sigma^2}y^TX_{\gamma}(X_{\gamma}^TX_{\gamma}+\frac{\sigma^2}{V^2}I_{\vert\gamma\vert\times\vert\gamma\vert})^{-1}X_{\gamma}^Ty,\nonumber
\end{eqnarray}
we focus on bounding the expression
\begin{equation}\small
y^T\left(X_{\gamma^\prime}(X_{\gamma^\prime}^TX_{\gamma^\prime}+\frac{\sigma^2}{V^2}I_{\vert\gamma^\prime\vert\times\vert\gamma^\prime\vert})^{-1}X_{\gamma^\prime}^T-X_{\gamma}(X_{\gamma}^TX_{\gamma}+\frac{\sigma^2}{V^2}I_{\vert\gamma\vert\times\vert\gamma\vert})^{-1}X_{\gamma}^T\right)y\nonumber
\end{equation}\normalsize
We have
\begin{eqnarray}
&&X_{\gamma^\prime}(X_{\gamma^\prime}^TX_{\gamma^\prime}+\frac{\sigma^2}{V^2}I_{\vert\gamma^\prime\vert\times\vert\gamma^\prime\vert})^{-1}X_{\gamma^\prime}^T-X_{\gamma}(X_{\gamma}^TX_{\gamma}+\frac{\sigma^2}{V^2}I_{\vert\gamma\vert\times\vert\gamma\vert})^{-1}X_{\gamma}^T\nonumber\\
&\preceq& P_{\gamma^\prime}-P_{\gamma}+\left(P_{\gamma}-X_{\gamma}(X_{\gamma}^TX_{\gamma}+\frac{\sigma^2}{V^2}I_{\vert\gamma\vert\times\vert\gamma\vert})^{-1}X_{\gamma}^T\right),
\end{eqnarray}
in the positive definite ordering. By Lemma \ref{lem:projappx} we have
\begin{eqnarray}
&&y^T\left(P_{\gamma}-X_{\gamma}(X_{\gamma}^TX_{\gamma}+\frac{\sigma^2}{V^2}I_{\vert\gamma\vert\times\vert\gamma\vert})^{-1}X_{\gamma}^T\right)y\\
&\leq&\frac{\sigma^2}{n(1-\delta)V^2+\sigma^2}\Vert y\Vert_{\lt}^2
\end{eqnarray}
On the other hand, we may expand
\begin{eqnarray}
y^T(P_{\gamma^\prime}-P_{\gamma})y&=&\left(X_{\gamma^0}\beta_{\gamma^0}^0\right)^T(P_{\gamma^\prime}-P_\gamma)X_{\gamma_0}\beta_{\gamma^0}^0\label{term1}\\
&&+2\left(X_{\gamma^0}\beta_{\gamma^0}^0\right)^T(P_{\gamma^\prime}-P_{\gamma})e\label{term2}\\
&&+e^T(P_{\gamma^\prime}-P_{\gamma})e,\label{term3}.
\end{eqnarray}
By applying Lemmas \ref{lem:term1}, \ref{lem:term2}, and \ref{lem:noise} to (\ref{term1}), (\ref{term2}), and (\ref{term3}) respectively, we obtain the bound
\begin{eqnarray*}
y^T(P_{\gamma^\prime}-P_\gamma)y&\leq&-n\left(1-\delta-\frac{\theta^2}{1-\delta}\right)\Vert\beta_{\gamma^0\setminus\gamma^\prime}^0\Vert_{\lt}^2\\
&&+4\sigma\left(1+\frac{\theta}{1-\delta}\right)\sqrt{(1+\alpha)Sn\log p}\Vert\beta_{\gamma^0\setminus\gamma^\prime}^0\Vert_{\lt}\\
&&+4\frac{\sigma^2}{1-\delta}(1+\alpha)S\log p
\end{eqnarray*}
Since $\gamma^\prime\in \{0,1\}_S^p\setminus G$, we may now employ Lemma \ref{lem:cleanup} and the fact that $\gamma^\prime\in\{0,1\}_S^p\setminus G$ to obtain the bounds
\begin{eqnarray*}
y^T(P_{\gamma^\prime}-P_\gamma)y&\leq&-\frac{n}{2}\left(1-\delta-\frac{\theta^2}{1-\delta}\right)\Vert\beta_{\gamma^0\setminus\gamma^\prime}^0\Vert_{\lt}^2+4\frac{\sigma^2}{1-\delta}(1+\alpha)S\log p\\
&\leq&\underbrace{\left(-\frac{A^2}{2}\left(1-\delta-\frac{\theta^2}{1-\delta}\right)+4\frac{\sigma^2}{1-\delta}\right)(1+\alpha)}_{= -2\sigma^2(1+\eta)}S\log p.
\end{eqnarray*}
The condition  $\delta<29/31$ in the statement of the theorem arises because we require $\eta>0$, or equivalenty $1+\eta>1$.  In particular, substituting $A=\frac{8\sqrt{2}}{1-\delta-\theta}$ above, we require 
\begin{eqnarray*}
\frac{30-32\delta+32\theta}{1-\delta}(1+\alpha)>1.
\end{eqnarray*}
Thus, it suffices to have $\alpha>0$ and $\delta<\frac{29}{31}$.
Accumulating the bounds we have established thus far, we have
\begin{eqnarray*}
\frac{1}{2}\left(\mu_{\gamma^\prime}^T\Sigma_{\gamma^\prime}^{-2}\mu_{\gamma^\prime}-\mu_{\gamma}^T\Sigma_{\gamma}^{-2}\mu_{\gamma}\right)&\leq& -(1+\eta) S\log p+\frac{1}{2}\frac{1}{n(1-\delta)V^2+\sigma^2}\Vert y\Vert_{\lt}^2,
\end{eqnarray*}
and hence
\begin{eqnarray*}
\frac{\det(\Sigma_{\gamma^\prime})}{\det(\Sigma_{\gamma})}e^{\frac{1}{2}(\mu_{\gamma^\prime}^T\Sigma_{\gamma^\prime}^{-2}\mu_{\gamma^\prime}-\mu_\gamma^T\Sigma_\gamma^{-2}\mu_\gamma)}&\leq&\left(\frac{n(1+\delta)V^2+\sigma^2}{n(1-\delta)V^2+\sigma^2}\right)^{S/2}e^{\frac{1}{n(1-\delta)V^2+\sigma^2}\frac{\Vert y\Vert_{\lt}^2}{2}}p^{-(1+\eta) S}
\end{eqnarray*}
for all $\gamma^0\leq\gamma$ and $\gamma^\prime\in\{0,1\}_S^p$. Therefore, we may bound (\ref{eq:UGdenominator}) from below by
\begin{eqnarray*}
&&\frac{1}{1+\left(\frac{n(1+\delta)V^2+\sigma^2}{n(1-\delta)V^2+\sigma^2}\right)^{S/2}e^{\frac{1}{n(1-\delta)V^2+\sigma^2}\frac{\Vert y\Vert_{\lt}^2}{2}}p^{-(1+\eta) S}\vert\{0,1\}_S^p\setminus G\vert}\\
&\geq&\frac{1}{1+\left(\frac{n(1+\delta)V^2+\sigma^2}{n(1-\delta)V^2+\sigma^2}\right)^{S/2}e^{\frac{1}{n(1-\delta)V^2+\sigma^2}\frac{\Vert y\Vert_{\lt}^2}{2}}p^{-(1+\eta)S}\binom{p}{S}}\\
&\geq&\frac{1}{1+\left(e^2\frac{n(1+\delta)V^2+\sigma^2}{n(1-\delta)V^2+\sigma^2}\right)^{S/2}e^{\frac{1}{n(1-\delta)V^2+\sigma^2}\frac{\Vert y\Vert_{\lt}^2}{2}}S^{-S}p^{-\eta S}}
\end{eqnarray*}
using the fact that $\binom{p}{S}\leq\left(\frac{ep}{S}\right)^S$. This completes the proof.
\end{proof}


\begin{thebibliography}{9}

\bibitem{Aka74}
\textsc{Akaike, H.} (1974). A new look at the statistical model identification. \textit{IEEE Transactions on Automatic Control} \textbf{19} 716--723.

\bibitem{ADLBS13}
\textsc{Armagan, A.}, \textsc{Dunson, D.} , \textsc{Lee, J.}, \textsc{Bajwa, W.}, and \textsc{Strawn, N.} (2013). Posterior consistency in linear models under shrinkage priors. \textit{Biometrika}, \textbf{100}, 1011-1018.

\bibitem{BCDH10}
\textsc{Baraniuk, R. G.}, \textsc{Cevher, V.}, \textsc{Duarte, M. F.|},and \textsc{Hegde, C.} (2010). Model-based compressive sensing. \textit{Information Theory, IEEE Transactions on}, \textbf{56}(4), 1982-2001.

\bibitem{BDDW08}
\textsc{Baraniuk, R. G.}, \textsc{Davenport, M.}, \textsc{DeVore, R.}, and \textsc{Wakin, M.} (2008). A simple proof of the restricted isometry property for random matrices. \textit{Constructive Approximation}, \textbf{28}(3), 253-263.

\bibitem{BLM12}
\textsc{Bayati, M.}, \textsc{Lelarge, M.}, and \textsc{Montanari, A.} (2012). Universality in polytope phase transitions and message passing algorithms. arXiv preprint arXiv:1207.7321.

\bibitem{BM11}
\textsc{Bayati, M.} and \textsc{Montanari, A.} (2011). The dynamics of message passing on dense graphs, with applications to compressed sensing. Information Theory, IEEE Transactions on, \textbf{57}(2), 764-785.

\bibitem{BI09}
\textsc{Berinde, R.} and \textsc{Indyk, P.} (2009, September). Sequential sparse matching pursuit. In Communication, Control, and Computing, 2009. Allerton 2009. 47th Annual Allerton Conference on (pp. 36-43). IEEE.

\bibitem{BD10}
\textsc{Blumensath, T.} and \textsc{Davies, M. E.} (2010). Normalized iterative hard thresholding: Guaranteed stability and performance. Selected Topics in Signal Processing, IEEE Journal of, \textbf{4}(2), 298-309.

\bibitem{Bontemps11}
\textsc{Bontemps, D.} (2011). Bernstein-Von Mises theorems for Gaussian regression with increasing number of regressors. \textit{Annals of Statistics} \textbf{39}, 2557--2584.

\bibitem{BDFKK11}
\textsc{Bourgain, J.}, \textsc{Dilworth, S.}, \textsc{Ford, K.}, \textsc{Konyagin, S.}, and \textsc{Kutzarova, D.} (2011). Explicit constructions of RIP matrices and related problems. Duke Mathematical Journal, \textbf{159}(1), 145-185.

\bibitem{CV12}
\textsc{Castillo, I.} and \textsc{Van Der Vaart, A.} (2012). Needles and Straw in a Haystack: Posterior concentration for possibly sparse sequences. The Annals of Statistics, 40(4), 2069-2101.

\bibitem{CP07}
\textsc{Cand\`{e}s, E. J. } and \textsc{Plan, Y.} (2007) Near-ideal model selection by $\ell_1$ minimization. \textit{Annals of Statistics} \textbf{37}, 2145--2177.

\bibitem{CRT05}
\textsc{Cand\`{e}s, E. J. }, \textsc{Romberg, J.}, and \textsc{Tao, T.} (2005). Stable signal recovery from incomplete and inaccurate measurements. \textit{Comm. Pure Appl. Math.} \textbf{59} 1207--1223. 

\bibitem{CT05}
\textsc{Cand\`{e}s, E. J.} and \textsc{Tao, T.} (2007). The Dantzig selector: statistical estimation when p is much larger than n. \textit{Annals of Statistics}, \textbf{35} 2313--2351.

\bibitem{Cev09}
\textsc{Cevher, V.} (2009). Learning with Compressible Priors. \textit{Proc. Neural Information Processing Systems}, Vancouver, B.C., Canada.

\bibitem{Donoho06}
\textsc{Donoho, D.} (2006). Compressed sensing. \textit{IEEE Transactions on Information Theory}, \textbf{52}, 1289--1306.

\bibitem{DT05}
\textsc{Donoho, D. L.}, and \textsc{Tanner, J.} (2005). Neighborliness of randomly projected simplices in high dimensions. Proceedings of the National Academy of Sciences of the United States of America, 102(27), 9452-9457.

\bibitem{DMM09}
\textsc{Donoho, D. L.}, \textsc{Maleki, A.}, and \textsc{Montanari, A.} (2009). Message-passing algorithms for compressed sensing. Proceedings of the National Academy of Sciences, 106(45), 18914-18919.

\bibitem{FM11}
\textsc{Fickus, M.} and \textsc{Mixon, D. G.} (2011, September). Deterministic matrices with the restricted isometry property. In Proc. of SPIE Vol (Vol. 8138, pp. 81380A-1).

\bibitem{FoucartRauhut2013}
\textsc{Foucart, S.} and \textsc{Rauhut, H.} (2013). A mathematical introduction to compressive sensing. \textit{Appl. Numer. Harmon. Anal. Birkh\"{a}user}, Boston.

\bibitem{GBK08}
\textsc{Gamper, U.}, \textsc{Boesiger, P.}, and \textsc{Kozerke, S.} (2008). Compressed sensing in dynamic MRI. Magnetic Resonance in Medicine, {\bf 59}(2), 365-373.

\bibitem{GR03}
\textsc{Ghosh, J. K.} and \textsc{Ramamoorthi, R. V.} (2003). \textit{Bayesian Nonparametrics}. Springer.

\bibitem{Ghosal99}
\textsc{Ghosal, S.} (1999). Asymptotic normality of posterior distributions in high-dimensional linear models. \textit{Bernoulli}, \textbf{5} 315--331.

\bibitem{GCD11}
\textsc{Gribonval, R.}, \textsc{Cevher, V.} and \textsc{Davies, M.} (2011) Compressible priors for high-dimensional statistics, submitted to \textit{IEEE Transactions on Information Theory}.

\bibitem{HZM09}
\textsc{Huang, J.}, \textsc{Zhang, T.}, and \textsc{Metaxas, D.} (2009). Learning with structured sparsity. In Proceedings of the 26th Annual International Conference on Machine Learning (pp. 417-424). ACM.

\bibitem{Iwen2009}
\textsc{Iwen, M. A.} (2009, March). Simple deterministically constructible RIP matrices with sublinear Fourier sampling requirements. In Information Sciences and Systems, 2009. CISS 2009. 43rd Annual Conference on (pp. 870-875). IEEE.

\bibitem{JM13a}
\textsc{Javanmard, A.} and \textsc{Montanari, A.} (2013). Confidence Intervals and Hypothesis Testing for High-Dimensional Regression. arXiv preprint arXiv:1306.3171.

\bibitem{JM13b}
\textsc{Javanmard, A.} and \textsc{Montanari, A.} (2013). Nearly Optimal Sample Size in Hypothesis Testing for High-Dimensional Regression. arXiv preprint arXiv:1311.0274.

\bibitem{Jiang07}
\textsc{Jiang, W.} (2007). Bayesian variable selection for high dimensional generalized linear models: Convergence rates of the fitted densities. \textit{The Annals of Statistics}, \textbf{35} 1487--1511.


\bibitem{LDP07}
\textsc{Lustig, M.}, \textsc{Donoho, D.}, and \textsc{Pauly, J. M.} (2007). Sparse MRI: The application of compressed sensing for rapid MR imaging. Magnetic Resonance in Medicine, 58(6), 1182-1195.

\bibitem{Lopes12}
\textsc{Lopes, M. E.} (2012). Estimating unknown sparsity in compressed sensing. arXiv preprint arXiv:1204.4227.

\bibitem{Mallat08}
\textsc{Mallat, S.} (2008). \textit{A wavelet tour of signal processing: the sparse way.} Access Online via Elsevier.


\bibitem{NT09}
\textsc{Needell, D.} and \textsc{Tropp, J. A.} (2009). CoSaMP: Iterative signal recovery from incomplete and inaccurate samples. Applied and Computational Harmonic Analysis, \textbf{26}(3), 301-321.


\bibitem{RFG12}
\textsc{Rangan, S.}, \textsc{Fletcher, A. K.}, and \textsc{Goyal, V. K.} (2012). Asymptotic analysis of MAP estimation via the replica method and applications to compressed sensing. \textit{Information Theory, IEEE Transactions on}, \textbf{58}(3), 1902-1923.

\bibitem{RV08}
\textsc{Rudelson, M.} and \textsc{Vershynin, R.} (2008). On sparse reconstruction from Fourier and Gaussian measurements. Communications on Pure and Applied Mathematics, 61(8), 1025-1045.

\bibitem{SY10}
\textsc{Saab, R.} and \textsc{Y�lmaz, \"{O}}. (2010). Sparse recovery by non-convex optimization-instance optimality. \textit{Applied and Computational Harmonic Analysis}, \textbf{29}(1), 30-48.

\bibitem{Sch65}
\textsc{Schwartz, L.} (1965). On Bayes procedure. \textit{Probability Theory and Related Fields}. \textbf{4} 10--26.

\bibitem{Tib96}
\textsc{Tibshirani, R.} (1996). Regression shrinkage and selection via the lasso. Journal of the Royal Statistical Society. Series B (Methodological), 267-288.

\bibitem{TG07}
\textsc{Tropp, J. A.} and \textsc{Gilbert, A. C.} (2007). Signal recovery from random measurements via orthogonal matching pursuit. \textit{Information Theory, IEEE Transactions on,} \textbf{53}(12), 4655-4666.
Chicago	

\bibitem{Tro06}
\textsc{Tropp, J. A.} (2006). Just relax: Convex programming methods for identifying sparse signals in noise. \textit{Information Theory, IEEE Transactions on}, \textbf{52}(3), 1030-1051.

\bibitem{Tro04}
\textsc{Tropp, J.} (2004). \textit{Topics in Sparse Approximation.}
Ph.D. dissertation, Univ. Texas at Austin.

\bibitem{vdGBR13}
\textsc{van de Geer, S.}, \textsc{B\"{u}hlmann, P.}, and \textsc{Ritov, Y. A.} (2013). On asymptotically optimal confidence regions and tests for high-dimensional models. arXiv preprint arXiv:1303.0518.

\bibitem{Wang2009}
\textsc{Wang, H.S.} (2009). Forward regression for ultra high-dimensional variable screening. \textit{Journal of the American Statistical Association} {\bf 104} 1512--1524 

\bibitem{Wang2011}
\textsc{Wang, T.} and \textsc{Zhu, L.X.} (2011). Consistent tuning parameter selection in high dimensional sparse linear regression. \textit{Journal of Multivariate Analysis}, \textbf{102} 1141--1151.

\bibitem{Ward2009}
\textsc{Ward, R.} (2009). Compressed sensing with cross validation. \textit{Information Theory, IEEE Transactions on}, \textbf{55}(12), 5773-5782.

\bibitem{Zhang2009}
\textsc{Zhang, T.} (2009). Some sharp performance bounds for least squares regression with $L_1$ regularization. \textit{The Annals of Statistics}, {\bf 37} 2109--2144.

\end{thebibliography}
\end{document}